\documentclass[11pt]{amsart}

\usepackage[margin=3cm]{geometry}
\usepackage{lipsum}
\usepackage{amsfonts}
\usepackage{graphicx}
\usepackage{algorithmic}
\usepackage{subfigure}
\usepackage{amsmath}     
\usepackage{amssymb}
\usepackage{mathtools}
\usepackage{bbm}
\usepackage{color}
\usepackage{siunitx}
\usepackage{hyperref}
\usepackage{cleveref}
\usepackage{amsaddr}

\hypersetup{
    colorlinks=true,
    linkcolor=blue,
    filecolor=magenta,      
    urlcolor=cyan,
}
\newtheorem{theorem}{Theorem}[section]
\newtheorem{remark}{Remark}[section]
\newtheorem{corollary}{Corollary}[section]

\newtheorem{assumption}{Assumption}
\newtheorem{example}{Example}

\def\COMMA{\,,}
\def\PERIOD{\,.}
\def\VAR{\mathrm{Var}}
\def\LR{\mathrm{LR}}
\def\CLR{\mathrm{CLR}}

\begin{document}
%%%%%%%%%%%%%%%%%%%%%%%%%%%%%%%%%%%%%%%%%%%%%%%%%%
%
% HEADER for article style
%\title{Steady state sensitivity analysis of continuous time  Markov chains}
%
% Authors for article style
%\author[1]{Ting Wang\thanks{tingw@udel.edu}}
%\author[1]{Petr Plech\'{a}\v{c}\thanks{plechac@udel.edu}}
%\affil[1]{Department of Mathematical Sciences, University of Delaware}

%%%%%%%%%%%%%%%%%%%%%%%%%%%%%%%%%
% 
% HEADER for amsart style
\title{Steady state sensitivity analysis of continuous time  Markov chains}
%
% Authors for amsart style
\author{Ting Wang and Petr {Plech\'{a}\v{c}}}
\address{Department of Mathematical Sciences \\University of Delaware\\ Newark, DE 19716}
\email{tingw@udel.edu, plechac@udel.edu}
%\urladdr{www.math.udel.edu/$\sim$tingw, www.math.udel.edu/\~{}plechac} 

%%%%%%%%%%%%%%%%%%%%%%%%%%%%%%%%%

%\date{\today}
%%\renewcommand\Authands{ and }

%----------------------------------------------------------
\begin{abstract}
In this paper we study Monte Carlo estimators based on the likelihood ratio approach for steady-state sensitivity.
 We first extend the result of Glynn and Olvera-Cravioto [doi: 10.1287/stsy.2018.002] to the setting of continuous time Markov chains with a countable state space which include models such as stochastic reaction kinetics and kinetic Monte Carlo lattice systems.  
 We show that the variance of the centered likelihood ratio estimators does not grow in time. This result suggests that the centered likelihood ratio estimators should be favored for sensitivity analysis when the mixing time of the underlying continuous time Markov chain is large, which is typically the case when systems exhibit multiscale behavior. We demonstrate a practical implication of this analysis with  two numerical benchmarks of biochemical reaction networks.
\end{abstract}

\maketitle

\section{Introduction}
Parametric sensitivity analysis (gradient estimation), which studies the sensitivity of an observable to variations in system parameters, has wide applications in sciences and engineering.
Oftentimes, parametric sensitivity analysis is performed in order to optimize certain gains or costs associated with a given system.
For example, in operations research such as queuing systems, the service rate is often optimized through sensitivity analysis to provide better customer service. 
%For instance, in machine learning the sensitivity analysis is an essential ingredient in optimizing
%the hyper-parameters through the maximum likelihood estimation procedure.
In this paper, our primary focus will be the steady-state sensitivity analysis of continuous time Markov chains (CTMCs) $\{X(t, c)\}_{t \geq 0}$ on a countable state space, where $c$ is a system parameter.
Systems that are modeled by such processes are often referred to as discrete event systems \cite{asmussen2007stochastic}.  
For $X(t, c)$ that is ergodic the steady-state expectation $\pi_c(f) \triangleq \sum_x f(x) \pi_c(x)$ is typically approximated by the ergodic average
\begin{equation}\label{eqn:ergodic-average}
  \pi_c(f) = \lim_{t\to\infty}\frac{1}{t}\int_0^t f(X(s,c))\,ds.
\end{equation}
Our goal is to compute the gradient of the steady-state expectation of an observable $f$ with respect to $c$ at $c = c^*$, i.e.,
\begin{equation*}
  \left.\frac{\partial}{\partial c}\right|_{c=c^*} \pi_c(f).
\end{equation*}
Note that we do not impose detailed balance assumption on $X(t, c)$.
%hence we are interested in the sensitivity at the non-equilibrium regime.

The steady-state sensitivity estimation is of particular importance in applications such as surface chemistry \cite{chatterjee2007overview}, and systems biology \cite{kitano2002computational, mcadams1997stochastic}, where evolution of the system state is modeled as a CTMC.
Several Monte Carlo based simulation methodologies have been designed for sensitivity analysis.
These approaches can be roughly classified into three categories.
%finite difference (FD), the infinitesimal perturbation analysis (IPA) and
%the likelihood ratio (LR). 
The first approach approximates the derivative $\partial_c\pi_c(f)$ by a finite difference scheme following from the Taylor expansion at $c^*+h$,
%which uses the finite difference quotient to approximate the derivative, i.e., 
\[
\left.\frac{\partial}{\partial c}\right|_{c=c^*} \pi_c(f) = \lim_{h \to \infty} \frac{\pi_{c^*+h}(f) - \pi_{c^*}(f)}{h},
\]
where the $\pi_{c^*}$ and $\pi_{c^*+h}$ are approximated by \eqref{eqn:ergodic-average}.
However, this approach is biased due to the truncation error in the finite difference approximation. Furthermore, small $h$ often leads to large variance of the estimator \cite{anderson2012efficient, rathinam2010efficient}.
In order to reduce the variance, the perturbed process $X(t, c^*+h)$ and 
the nominal process $X(t, c^*)$ are often coupled in suitable ways
\cite{anderson2012efficient,glynn2017likelihood,rathinam2010efficient}.
The second approach is the infinitesimal perturbation analysis (IPA) \cite{glasserman1991gradient}, which involves taking a derivative of the sample path. 
%, i.e., 
%\[\left.\frac{\partial}{\partial c}\right|_{c=c^*} \pi_c(f) = \lim_{t \to \infty}\left.\frac{\partial}{\partial c}\right|_{c=c^*} \frac{1}{t}\int_0^t f(X(s,c))\,ds.\]
Often the regenerative structure is assumed and the sensitivity is computed through differentiating 
the regenerative ratio \cite{glasserman1991strongly}. 
%Here we assume the strong consistency of the IPA approach, which often requires checking
%stringent sufficient conditions in a case-by-case manner \cite{glasserman1991strongly}.
The third approach which is also the focus of our study is the likelihood ratio (LR) method (also known as the Girsanov transform method)
\cite{ glynn1990likelihood,plyasunov2007efficient,wang2016efficiency, warren2012steady}.
Instead of the sample path differentiation as in IPA, LR differentiates the Radon-Nikodym derivative $L(t,c)$ between the nominal path-space measure $\mathbb{P}_{c}$ and the perturbed path-space measure $\mathbb{P}_{c^*}$.

Sensitivity analysis based on the LR have been studied extensively for the finite time horizon case \cite{glynn1990likelihood,plyasunov2007efficient,wang2016efficiency}.
However, in the steady-state case, the LR approach has been mainly used in a heuristic manner.
For example, in \cite{arampatzis2016efficient,hashemi2016stochastic}, the LR estimator
\[
\frac{1}{t}\int_0^t f(X(s,c^*)) \,ds \left.\frac{\partial}{\partial c}\right|_{c=c^*}L(t, c)
\]
is applied to approximate the sensitivity.
Furthermore, $\pi_{c^*}(f)$ is used as a control variate to reduce the variance which leads to 
the centered likelihood ratio (CLR) estimator
\[
\frac{1}{t}\int_0^t (f(X(s,c^*)) - \pi_{c^*}(f)) \,ds \left.\frac{\partial}{\partial c}\right|_{c=c^*}L(t, c)\,.
\]

\medskip

In this paper we address the following questions
regarding the LR approach to the steady-state sensitivity.
First, do the LR estimators converge to the correct sensitivity and what is the notion of convergence?
Second, what are the general sufficient conditions for the convergence?
Finally, the main contribution of the paper, how does the estimator's variance grow in terms of time $t$?

Recently, Glynn and Olvera-Cravioto \cite{glynn2017likelihood}
have answered the first two questions for the discrete time Markov chain (DTMC) using the Poisson equation approach, which has a natural link with the centered LR estimator.
Following the same approach, we first present an extension of their results to the CTMC setting. Oftentimes, extending results from the DTMC to those for the CTMC is through constructing and analyzing the embedded chain. 
However, thanks to the integration by parts formula for semimartingales and the Burkholder--Davis--Gundy (BDG) inequality, we find it convenient to establish the results directly in the CTMC setting. 
The fundamental assumption for studying the long time behavior of CTMC is the Foster-Lyapunov drift condition \cite{MT1993stability}, which is also the main
assumption in this work.
These techniques allow us to address the first two questions, answers to which are essentially an extension of the results in \cite{glynn2017likelihood}.
Furthermore, in the appendix, we provide a verifiable condition for the existence of Lyapunov function.
Finally, the answer to the last question is of 
particular importance from the computational perspective
since numerical methods that accumulate variance along the history of a simulation often require a large number of samples to obtain statistically reliable results.
In \cite{arampatzis2016efficient, hashemi2016stochastic}, it has been numerically observed that the CLR estimator tends to have a constant variance in time.
In this work, we show that this numerical observation can be rigorously justified under the mathematical framework of this paper.

\medskip

%\subsection{Outline of the paper}
This article is organized as follows. 
In Section~\ref{sec:LR-sensitivity} we present the mathematical 
framework for the LR approach and set up technical assumptions
in the CTMC setting.
In Section~\ref{sec:unbiased-estimator} we present the first main result proving
that the derived LR estimators for steady-state sensitivity are asymptotically unbiased. 
In Section~\ref{sec:weak-limit} we show that the centered LR estimators have well-defined limiting distributions whereas the standard LR estimators do not. 
In Section~\ref{sec:var-analysis}, we prove that the order of LR estimators grows like $\mathcal{O}(t)$ while that of the {\it centered} version {\it does not} grow in $t$. 
This suggests that the centered LR estimators are suitable for situations when long time simulations are needed (e.g., systems involving rare events).
Finally in Section~\ref{sec:numerics},
we provide numerical examples where a variance reduction up to the order of $10^4$ to $10^5$ is observed by using the centered LR estimators.

%-----------------------------------------------------------------------------------------

\section{Mathematical framework}\label{sec:LR-sensitivity}
\subsection{The continuous time Markov chains and the sensitivity problem}
We study models described by a CTMC $\{X(t)\}_{t \geq 0}$ defined on a probability space $(\Omega, \mathcal{F}, \mathbb{P})$ with values in a countable state space $E$. 
Let $\mathcal{F}_t$ be the natural filtration generated by the process up to time $t$. 
Given an initial distribution, the CTMC is completely characterized by its associated semigroup $\{T(t)\}$ which is defined as
\[(T(t) f) (x) \triangleq \mathbb{E}^x\{f(X(t))\}\]
for any measurable $f: E \to \mathbb{R}$ and the initial state $x \in E$. 
The generator of the semigroup $\{T(t)\}$ is the linear operator $\mathcal{L}$
defined by 
\[
\mathcal{L}f = \lim_{t \to 0}\frac{1}{t} \{T(t)f - f\}
\]
for all measurable $f:E\to \mathbb{R}$ such that the above limit exists.
Since the state space $E$ is discrete, the generator $\mathcal{L}$ can be
equivalently considered 
as an infinite dimensional matrix whose row $x$ and column $y$ entry $\mathcal{L}(x, y)$ specifies the transition rate for $X(t)$ to transition from the state $x$ to the state $y$. Hence, the generator is also known as the transition (jump) rate matrix or $Q$-matrix. 

As an application of the presented analysis to such a CTMC-based model we consider stochastic reaction networks \cite{gillespie2007stochastic}.
Generally, a stochastic reaction network which consists of $m$ reactions and $n$ species $\{S_1, \cdots, S_n\}$ is of the form
\[
\sum_{i = 1}^n \nu_{ij}^{-} S_i \to \sum_{i=1}^n \nu_{ij}^{+}S_i, \qquad j =  1, \ldots, m\,.
\]
Here $\nu_{ij}^{+}$ and $\nu_{ij}^-$ are nonnegative integers that correspond to the number of $S_i$ molecules
that are produced and consumed respectively in the $j$-th reaction channel. 
Hence the stoichiometric vector (i.e., the net change) with respect to the $j$-th reaction is
\[\nu_j = \nu_j^+ - \nu_j^-.\]
The $n$-dimensional state vector $X(t, c)$ characterizes the state of the system
where each entry $X_i(t, c)$ represents the number of molecules of the species
$S_i$ at time $t$ and $c$ is a vector of parameters.
  We assume that $X$ is {\em c\`adl\`ag}, i.e.\ paths of $X$ are right-continuous with left limits and hence, if the reaction channel $j$ fires at time $t$, then $X(t, c)=X(t-, c)+\nu_j$. 
For $j =1,\ldots,m$ we denote by $R_j(t, c)$ the number of firings of the $j$-th
reaction channel during $(0,t]$. 
%Thus $X(t, c) = X(0, c) + \nu R(t, c)$ for $t \geq 0$,
%  where $\nu$ is the stoichiometric matrix whose $j$th column is $\nu_j$
%  and $R(t, c)=(R_1(t, c),\ldots,R_m(t, c))^T$. 
We note that $R(0)=0$ and $R_j(t, c)-R_j(t-, c)$ is either $0$ or $1$ depending on whether the $j$-th reaction occurs at time $t$.

The process $X(t, c)$ is assumed to be Markovian on the probability space 
$(\Omega, \mathcal{F}, \mathbb{P})$ and
associated with each reaction channel is an intensity function (jump rate) $a_j(x, c), j = 1, \ldots, m$, which is such that 
\[\lim_{h\to 0}\frac{1}{h}\mathbb{P}(X(t+h, c) = x + \nu_j | X(t, c) = x) = a_j(x, c).  
\]
From the intensity functions we can construct the generator matrix $\mathcal{L}_c$ of the Markov chain $X(t,c)$ such that 
\begin{equation}\label{eqn:generator-1}
\mathcal{L}_c(x, y) = \begin{cases}
a_j(x, c), & y = x + \nu_j;\\
-\sum_{j=1}^m a_j(x, c), & y = x;\\
0                        & \text{otherwise.}
\end{cases}
\end{equation}
By a slight abuse of notation, $\mathcal{L}_c$ will also denote the infinitesimal generator operator of the process $X(t, c)$, 
which is the linear operator of the form
\begin{equation}\label{eqn:generator-2}
(\mathcal{L}_c f)(x) = \sum_{j=1}^m a_j(x, c) \Delta_j f(x),
\end{equation}
where $\Delta_j f: E \to \mathbb{R}$ is the
$j$-th increment function
\begin{equation}\label{eqn:increment-function}
\Delta_j f(x) \triangleq f(x + \nu_j) - f(x),
\end{equation}
which is the increment to the observable $f$ due to one $j$-th reaction.
For a given observable $f$ and the generator $\mathcal{L}_c$,
the solution $\hat{f}_c$ to the Poisson equation
\begin{equation}\label{eqn:Possion}
-\mathcal{L}_{c}\hat{f}_c = f - \pi_c(f)
\end{equation}
is used to construct 
\begin{equation}\label{eqn:M}
    M(t, c) \triangleq \hat{f}(X(t, c)) - \hat{f}(X(0, c)) - \int_0^{t} \mathcal{L}_c\hat{f}_c(X(s, c)) \,ds,
\end{equation}
which is an $\mathcal{F}_t$-local martingale by Dynkin's lemma \cite{ethier2009markov}.
Finally, for $j = 1,\ldots, m$ we define another $\mathcal{F}_t$-local martingale 
\begin{equation}\label{eqn:Y}
    Y_j (t) = R_j(t) - \int_0^t a_j(X(s))\, ds
\end{equation}
provided that $\int_0^t a_j(X(s))\, ds < \infty$ for all $t$ \cite{bremaud1981point}.

Often we are interested in estimating the steady-state sensitivity 
\[
\nabla_{c=c^*} \pi_c(f),
\]
where $f: E \to \mathbb{R}$ is an observable.
Without loss of generality, 
we only consider the partial derivative with respect to $c_1$ (the first component of the vector $c$), 
i.e., 
\begin{equation}\label{eqn:sensitivity}
\left.\frac{\partial}{\partial c_1}\right|_{c_1=c_1^*} \pi_c(f)\PERIOD
\end{equation}
\begin{remark}
{\rm 
For the ease of exposition, from now on, we slightly abuse the notation by writing $c_1$ as $c$ in the above definition of sensitivity. 
Furthermore, unless otherwise specified, we omit $c^*$ when we write quantities dependent on $c$ and  evaluated at  $c^*$, e.g., $X(t, c^*)$, is abbreviated to $X(t)$ or $\pi_{c^*}(f)$ becomes simply $\pi(f)$, etc.
}
\end{remark}

\subsection{Likelihood ratio estimators}
We formally derive the LR estimators in this section. 
%Note that the process ${X(t,c)}_{t\geq 0}$ can be viewed 
Let $\mathbb{P}_c$ be the path-space probability measure induced by the process $X(t, c)$.
Suppose that $\mathbb{P}_c$ is absolutely continuous with respect to $\mathbb{P}_{c^*}$.
Then we can define their Radon-Nikodym derivative 
\[
L(t, c) \triangleq \left.\frac{d\mathbb{P}_c}{d\mathbb{P}_{c^*}}\right|_{\mathcal{F}_t}
\]
as the LR process. 
Hence by definition
\[
\mathbb{E}_c\left\{\frac{1}{t}\int_0^t f(X(s)) \,ds\right\} = \mathbb{E}_{c^*}\left\{\frac{1}{t}\int_0^t f(X(s)) \,ds \,L(t, c)\right\}\COMMA
\]
where $\mathbb{E}_c$ is the expectation corresponding to the path-space measure $\mathbb{P}_c$.
Assuming that we can differentiate inside the expectation $\mathbb{E}_{c^*}$ we have 
\begin{equation}\label{eqn:swap-E-diff}
    \left.\frac{\partial}{\partial c}\right|_{c=c^*}\mathbb{E}_c\left\{\frac{1}{t}\int_0^t f(X(s)) \,ds\right\} = \mathbb{E}_{c^*}\left\{\frac{1}{t}\int_0^t f(X(s)) \,ds \,Z(t)\right\}\COMMA
\end{equation}
where 
$Z(t) \triangleq \left.\frac{\partial}{\partial c}\right|_{c=c^*}L(t, c)$ is the weight process for the sensitivity. 
The explicit formula for $Z(t)$ can be easily derived as
 \begin{equation}\label{eqn:Z-full-form}
 \begin{split}
 Z(t)  = \sum_{j = 1}^m \int_0^t \frac{\frac{\partial a_j}{\partial c}(X(s-))}{a_j(X(s-))} dR_j(s) - \sum_{j = 1}^m \int_0^t\frac{\partial a_j}{\partial c}(X(s)) ds\PERIOD
\end{split} 
 \end{equation}
and it is, in general,  a zero mean $\mathcal{F}_t$-local martingale.

The idea of the LR approach is simply using the right hand side of \eqref{eqn:swap-E-diff} to approximate the sensitivity.
Hence, the standard LR estimator is 
\begin{equation}\label{eqn:LR}
    \mathcal{S}_{\LR}(t) = \frac{1}{t}\int_0^t f(X(s))ds\, Z(t)\PERIOD
\end{equation}
By the standard Monte Carlo method we sample the above LR estimator for a fixed time $t$ and 
then use the sample average to approximate the sensitivity. 
Since the efficiency of a Monte Carlo simulation depends on the estimator variance, 
in order to reduce the variance one can apply the control variate approach \cite{asmussen2007stochastic}, which leads to the CLR estimator 
    \begin{equation}\label{eqn:CLR}
      \mathcal{S}_{\CLR}(t) = \frac{1}{t}\int_0^t \left(f(X(s)) - \pi(f)\right)ds\, Z(t)\PERIOD
    \end{equation}
It can be shown that $\pi(f)$ is the optimal control variate coefficient \cite{glynn2017likelihood}.
Note that by properties of the conditional expectation and the fact that $Z(t)$ is a zero mean $\mathcal{F}_t$-local martingale,  
\[
\mathbb{E}\left\{\int_0^t f(X(s)) (Z(t) - Z(s)) ds\right\} = \int_0^t\mathbb{E}\left\{ f(X(s)) \mathbb{E}\{(Z(t) - Z(s)) | \mathcal{F}_s \}\right\}ds = 0\PERIOD
\]
Hence, we derive an alternative form of the LR estimator \cite{glynn2017likelihood}, which we refer to as
the integral type LR estimator (int-LR)
\begin{equation}\label{eqn:new-LR}
    \tilde{\mathcal{S}}_{\LR}(t) = \frac{1}{t}\int_0^t f(X(s)) Z(s) \,ds \PERIOD
\end{equation}
The centering trick can be used as well to derive the integral type CLR estimator (int-CLR)
\begin{equation}\label{eqn:new-CLR}
    \tilde{\mathcal{S}}_{\CLR}(t) = \frac{1}{t}\int_0^t (f(X(s)) - \pi(f)) Z(s)\, ds\PERIOD
\end{equation}
We will focus on comparing these four estimators from several different perspectives in the next few sections. 

\begin{remark}
{\rm
For ease of presentation, from now on we assume that the intensities are of the form 
$a_j(c, x) = c_j b_j(x)$ for $j = 1, \dots, m$, 
where $b_j : E \to \mathbb{R}_+$.
%Note that this form is the typical form of the intensity for most types of stochastic
%reaction systems. 
Note that the mass action type intensity is of this form.
In this case, the weight process (with respect to $c_1$) \eqref{eqn:Z-full-form} can be simply written as 
\begin{equation}\label{eqn:Z}
    Z(t) \triangleq \frac{1}{c_1}\left(R_1(t) - \int_0^t a_1(X(s))\, ds\right),
\end{equation}
which is a zero mean $\mathcal{F}_t$-martingale if $\mathbb{E}_{c^*}\{\int_0^t a_1(X(s)) \, ds\} < \infty$ \cite{bremaud1981point}.
We point out that our results presented here do not rely on the specific form 
of the intensity. 
For a general form of the intensity as in \eqref{eqn:Z-full-form}, one just needs to impose certain regularity assumptions on $\frac{\partial a_j}{\partial c}(X(s))$ and the same analysis presented in the following sections carries out.
}
\end{remark}

\subsection{Assumptions and preparatory results}
To facilitate the analysis we first introduce the $V$-norm. 
For a signed measure $\mu$ we denote the integral of any function $f$ against $\mu$ to be
\begin{equation}\label{eqn:mu-integral}
    \mu(f) = \sum_{x \in E} f(x) \mu(x)\PERIOD
\end{equation}
For a given function $V: E \to [1, \infty)$, we define the $V$-norm of a measurable function $f$ as 
\begin{equation}\label{eqn:V-norm-function}
    |f|_V = \sup_{x \in E} \frac{|f(x)|}{V(x)}\PERIOD
\end{equation}
Similarly, we define the $V$-norm of the measure $\mu$ to be
\begin{equation}\label{eqn:V-norm-measure}
    \|\mu\|_V = \sup_{f: |f| \leq V} |\mu(f)|\PERIOD
\end{equation}
Finally, the $V$-norm of a linear operator $\mathcal{L}$ on $E \times E$ is defined by
\begin{equation}\label{eqn:V-norm-generator}
    \|\mathcal{L}\|_V = \sup_{x \in E} \frac{1}{V(x)} \sum_{y\in E} |\mathcal{L}(x, y)| V(y)\PERIOD
\end{equation}
We make the following assumptions throughout the paper.
\begin{assumption}\label{assume:irreducible}
For any $c \in B_{\epsilon}(c^*)$, the CTMC $X(t, c)$ is irreducible.
\end{assumption}

It is often difficult to verify the irreducibility condition due to the infinity of
the state space $E$ and the presence of conservation relations among species populations.
Nevertheless, we point out that a linear algebraic procedure has been developed recently by Gupta and Khammash \cite{gupta2018computational} to identify irreducible classes in a systematic way.

\begin{assumption}\label{assume:drift-condition}
There exist a norm like function $V: E \to [1, \infty)$, a finite set $D$,
and some constants $\alpha_1 >0$ and $ 0<\alpha_2 < \infty$ such that 
\begin{equation}\label{eqn:V-drift}
    \mathcal{L}_c V(x) \leq -\alpha_1 V(x) + \alpha_2\mathbbm{1}_D(x), 
    \qquad x \in E,
\end{equation}
for all $c \in B_{\epsilon}(c^*)$.
\end{assumption}

The Foster-Lyapunov inequality \eqref{eqn:V-drift} in Assumption~\ref{assume:drift-condition} is a criterion for proving ergodicity of a CTMC. 
The following theorem is a direct consequence of Theorems $4.2$ and $7.1$ in \cite{MT1993stability}.
\begin{theorem}\label{thm:V-ergodic}
Given Assumption~\ref{assume:drift-condition}, for any fixed $c \in B_{\epsilon}(c^*)$ and initial state $x \in E$, there exist a unique probability measure $\pi_c$, a constant $\gamma< 0$, and a function $C(x) < \infty$ such that 
\begin{equation}
    \|P_t^c(x, \cdot) - \pi_c\|_V \leq C(x) e^{\gamma t}
\end{equation}
for all $t > 0$, where $P_t^c(x, \cdot) = \mathbb{P}^x(X(t, c) \in \cdot)$ is the distribution of $X(t, c)$ that starts at $x$.
Moreover, $X(t, c)$ is positive recurrent and $\pi_c (V)$ is finite.
\end{theorem} 

\begin{remark}
{\rm 
The above theorem implies that $X(t, c)$ is ergodic under the $V$-norm, i.e., for any function $g: E \to \mathbb{R}$ with bounded $V$-norm (i.e., $|g|_{V} < \infty$),
\begin{equation}\label{eqn:V-ergodic}
    \mathbb{E}_x g(X(t, c)) \to \pi_c(g) 
\end{equation}
as $t \to \infty$.
}
\end{remark}

%{\bf Drift-Diffusivity Condition} \cite{gupta2014scalable}: 
%Suppose that for a positive vector $v \in \mathbb{R}^n$, there exist positive constants $\alpha_1, \alpha_2, \alpha_3, \alpha_4$ and a nonnegative constant $\alpha_5$ such that 
%\begin{equation}
%\sum_{j = 1}^m a_j(c, x) \langle v, \eta_j \rangle \leq \alpha_1 - \alpha_2 \langle v, x \rangle
%\end{equation} 
%and
%\begin{equation}
%\sum_{j = 1}^m a_j(c, x) \langle v, \eta_j \rangle^2 \leq \alpha_3 + \alpha_4 \langle v, x \rangle + \alpha_5 \langle v, x \rangle^2
%\end{equation}
%for all $x \in E$ and $c \in B_{\epsilon}(c^*)$. 
%Note that we require that the DD condition holds globally for all $c \in B_{\epsilon}(c^*)$ which is stronger than that in \cite{gupta2014scalable} where the condition only holds for $c^*$.
%
%Now choose $V(x)=1 + \langle v, x \rangle$ so that by the first part of the DD condition
%\begin{equation}\label{eqn:V-drift}
%\mathcal{L}V(x) = \sum_{j=1}^m a_j(c, x) \langle v, \eta_j \rangle \leq \alpha_1 + \alpha_2 -\alpha_2 V(x). 
%\end{equation} 
\begin{assumption}\label{assume:intensity-bounded}
For each $j = 1, \ldots, m$, $|a_j|_{\sqrt{V}} < \infty$, i.e., the intensities are $\sqrt{V}$-bounded.
\end{assumption}

\begin{remark}
{\rm 
Note that this assumption implies that $a_j(x)$  is $V$-bounded as well since $\sqrt{V} \leq V$.
Hence, by Theorem~\ref{thm:V-ergodic},  
\begin{equation}\label{eqn:intensity-ergodic}
    \mathbb{E} a_j(X(t)) \to \pi(a_j)
\end{equation} 
as $t \to \infty$.
}
\end{remark}

We assume the following regularity condition on the infinitesimal
generators. 
\begin{assumption}\label{assume:generator}
$\lim_{h \to 0} \| \mathcal{L}_{c^*} - \mathcal{L}_{c^* + h}\|_V = 0.$
\end{assumption}
\begin{remark}
{\rm 
Since the intensities are assumed to be of the form $a_j(x, c) = c_j b_j(x)$ and the perturbation is only in $c_1$, it can be readily verified that Assumption~\ref{assume:generator} is equivalent to the following regularity condition:
\begin{equation}\label{eqn:generator-regularity}
    \sup_{x\in E} b_1(x) \left(\frac{V(x+\nu_1)}{V(x)} + 1\right) < \infty\PERIOD
\end{equation}
}
\end{remark}

In order to carry out rigorous analysis we also impose regularity conditions on observables. 
\begin{assumption}\label{assume:observable-bounded}
The observable $f$ is $\sqrt{V}$-bounded, i.e.,  $|f|_{\sqrt{V}} < \infty$.
\end{assumption}
%Since $\pi(V) < \infty$ by Theorem \ref{thm:V-ergodic}, we have $\pi(f^2)$ is finite by the above assumption.

\begin{assumption}\label{assume:increment-bounded}
There exists a constant $K$ such that $|\Delta_j f(x)| \leq K$ for $j = 1, \ldots, m$ and $x \in E$.
\end{assumption}

\begin{remark}
{\rm 
Under Assumption~\ref{assume:increment-bounded}, it follows that the increment function $\Delta_j f$ satisfies
\begin{equation}\label{eqn:Lipschitz}
    |\Delta_j f(x)| = |f(x + \nu_j) - f(x)| \leq K  \leq  K  V^{\frac{1}{4}}(x)
\end{equation}
for $j = 1, \ldots, m$ and $x \in E$.
That is, each increment function $\Delta_j f$ is $V^{1/4}$-bounded.
The reason why we choose $V^{1/4}$ to bound the increment functions will be made clear in the proof of 
Theorem~\ref{thm:joint-weak-convergence}.
Note that this assumption does not appear in the DTMC setting in \cite{glynn2017likelihood}. 
However, this assumption seems to be
necessary for us to carry out the analysis in the continuous time setting and carry through rigorous proofs. 
Nevertheless, this assumption is reasonable since it includes most observables such as the indicator function and Lipschitz functions.
}
\end{remark}

\begin{example}
We give a concrete example that satisfies the above assumptions. 
Consider the simple reversible isomerization network 
\[S_1 \xrightarrow{c_1} S_2, \quad S_2 \xrightarrow{c_1} S_1\]
with initial population $[1, 0]$. 
The system is clearly irreducible. 
Furthermore, it can be readily verified that the function $V(x) = 1 + x_1 x_2$ satisfies 
Assumptions~\ref{assume:drift-condition} and~\ref{assume:intensity-bounded}. 
The regularity condition \eqref{eqn:generator-regularity} is trivial in this example since the state space $E = \{[0, 1], [1, 0]\}$ is finite. 
Finally, for observables $f$ that are subquadratic and Lipschitz (e.g., $f(x) = \sin x$), Assumptions~\ref{assume:observable-bounded} and~\ref{assume:increment-bounded} are trivially satisfied as well.
\end{example}

Before we proceed to the next section, we briefly discuss the above running assumptions 
in the setting of the Hill dynamics. 
In the literature of stochastic chemical kinetics there exists another important form of the intensity function 
\[
a_j(x, c) = \frac{x_j^n}{c_j + x_j^n}\COMMA
\]
which is known as the form of the Hill dynamics.
Note that the above assumptions are more verifiable in the Hill dynamics setting because each intensity $a_j$ is uniformly bounded.
For instance, the regularity condition~\eqref{eqn:generator-regularity} becomes 
\[
\sup_{x \in E} \frac{V(x+\nu_1)}{V(x)} < \infty,
\] 
which is more natural than that of the mass action setting.

\subsection{The solution to the Poisson equation}
Before we start the analysis of the four LR estimators
we need the following important result regarding the solution of the Poisson equation \eqref{eqn:Possion}.
This result will be used frequently in what follows. However, we omit the proof since it involves concepts that are beyond the scope of this paper. 
An interested reader is referred to \cite{glynn1996liapounov} for details of the proof.

%Using this lemma the following theorem is an adaptation of Theorem $3.2$ in \cite{glynn1996liapounov}. 
%Although this result will be used frequently in the sequel we omit the detailed proof since it involves concepts that are beyond the scope of this paper. 
%In fact, we only need the condition \eqref{eqn:drift-resolvent-form} to show the resolvent chain satisfies the 
%desired properties and then lift the result to the continuous time chain. 
%An interested reader is referred to \cite{glynn1996liapounov} for details of the proof. 
%This result will be used frequently in the sequel.

\begin{theorem}\label{thm:Poisson-solution}
Suppose that the Foster-Lyapunov drift condition \eqref{eqn:V-drift} in Assumption~\ref{assume:drift-condition} holds for a function $V$.
Then $X(t)$ is positive recurrent.
Furthermore, for each fixed $p \in \mathbb{Z}_+$, we have $\pi (V^{1/p}) < \infty$, and 
for any function $g: E \to \mathbb{R}$ satisfying $|g|_{V^{1/p}} < \infty$, the Poisson equation \eqref{eqn:Possion}
admits a solution $\hat{g}$ satisfying 
$|\hat{g}(x)| \leq K V(x)^{1/p}$ for some constant $K > 0$.
\end{theorem}

\section{Asymptotic unbiasedness of LR estimators}\label{sec:unbiased-estimator}
In this section we show that all four LR estimators are asymptotically unbiased in the sense that
they converge to the correct sensitivity in expectation. 
The technique of the proof is similar to that in the DTMC setting \cite{glynn2017likelihood}. However, the integration by parts formula for semimartingales and the BDG inequality allow us to carry out the analysis more conveniently.

\begin{theorem}\label{thm:unbiased-estimator}
The LR estimator is asymptotically unbiased, i.e.,  
\[\mathbb{E}\left\{\frac{1}{t}\int_0^t f(X(s))ds\, Z(t) \right\} \to \left.\frac{\partial}{\partial c}\right|_{c=c^*} \pi_c(f)\]
as $t \to \infty$\PERIOD
\end{theorem}

\begin{proof}
Since $Z(t)$ is a zero mean martingale we have
\[ 
\mathbb{E}\left\{\frac{1}{t}\int_0^t f(X(s)) \, ds\, Z(t) \right\} = \mathbb{E}\left\{\frac{1}{t}\int_0^t (f(X(s)) - \pi(f))\, ds\, Z(t) \right\}\PERIOD
\]
Using the Poisson equation \eqref{eqn:Possion} the right-hand side of the above equation can be written as
\begin{equation*}
\mathbb{E}\left\{\frac{1}{t}\int_0^t (f(X(s)) - \pi(f))\, ds\, Z(t) \right\}= -\mathbb{E}\left\{\frac{1}{t}\int_0^t \mathcal{L}_{c^*} \hat{f}(X(s))\, ds \, Z(t)\right\}\PERIOD
\end{equation*}
We observe that
\begin{equation}\label{eqn:martingale-product}
\begin{split}
    -\mathbb{E}\left\{ \frac{1}{t}\int_0^t \mathcal{L}_{c^*} \hat{f}(X(s))\, ds \, Z(t)\right\} =
    & \mathbb{E}\left\{\frac{1}{t}\left(\hat{f}(X(t)) - \hat{f}(X(0)) - \int_0^t \mathcal{L}_{c^*} \hat{f}(X(s))\, ds\right)\, Z(t)\right\}\\
    & - \mathbb{E}\left\{\frac{1}{t}\hat{f}(X(t)) Z(t)\right\}\PERIOD
\end{split}
\end{equation}
By the Cauchy-Schwarz inequality we obtain 
\begin{equation}\label{eqn:tail-term}
\mathbb{E}\left\{\frac{1}{t}\hat{f}(X(t))Z(t)\right\} \leq \mathbb{E}\left\{ \frac{1}{t} \hat{f}(X(t))^2\right\}^{1/2} \mathbb{E}\left\{ \frac{1}{t} Z(t)^2\right\}^{1/2}\PERIOD
\end{equation}
We claim that the right-hand side of \eqref{eqn:tail-term} vanishes as $t\to \infty$. 
In fact, by the BDG inequality (see Appendix~\ref{app_BDG}), there exists a constant $K$ such that 
\[
\frac{1}{t}\mathbb{E}\left\{Z(t)^2\right\} \leq \frac{K}{c_1^2t}\mathbb{E}R_1(t) = \frac{K}{c_1^2}\mathbb{E} \left\{\frac{1}{t}\int_0^t a_1(X(s))\, ds\right\}\PERIOD 
\]
Recalling Assumption~\ref{assume:drift-condition} and its consequence \eqref{eqn:intensity-ergodic}, we immediately have 
\[
\frac{K}{c_1^2}\mathbb{E} \left\{\frac{1}{t}\int_0^t a_1(X(s))\, ds\right\} \to \frac{K}{c_1^2} \pi (a_1) < \infty
\]
as $t \to \infty$. 
For the other term of \eqref{eqn:tail-term}, by Assumption~\ref{assume:generator} and Theorem~\ref{thm:Poisson-solution}, 
the solution to the Poisson equation, denoted by $\hat{f}$, satisfies $|\hat{f}^2|_V < \infty$ and hence
\[
\mathbb{E}\left\{\hat{f}(X(t))^2\right\} \to \pi (|\hat{f}|^2) < \infty\PERIOD 
\]
Taking $t \to \infty$ leads to 
\[
\mathbb{E}\left\{ \frac{1}{t} \hat{f}(X(t))^2\right\} \to 0\PERIOD 
\]
Therefore, it is sufficient to show that 
\[
\lim_{t\to\infty}  \mathbb{E}\left\{\frac{1}{t}\left(\hat{f}(X(t)) - \hat{f}(X(0)) - \int_0^t \mathcal{L}_{c^*} \hat{f}(X(s))\, ds\right)\, Z(t)\right\} =  \left.\frac{\partial}{\partial c}\right|_{c=c^*} \pi_c(f)\PERIOD
\]
We recall the It\^{o} formula for jump processes 
\cite{rogers1994diffusions}
\[
\hat{f}(X(t)) = \hat{f}(X(0)) + \sum_{j=1}^m \int_0^{t} \hat{f}(X(s-) + \nu_j) - \hat{f}(X(s-)) \, dR_j(s)
\]
and  also that $\mathcal{L}_{c^*} \hat{f}(x) = \sum_{j=1}^m a_j(x) (\hat{f}(x+\nu_j) - \hat{f}(x))$.
Thus we have 
\[
\hat{f}(X(t)) - \hat{f}(X(0)) - \int_0^t \mathcal{L}_{c^*} \hat{f}(X(s))\, ds = \sum_{j = 1}^m \int_0^t (\hat{f}(X(s-) + \nu_j) - \hat{f}(X(s-))) \, dY_j(s)\COMMA 
\]
where $Y_j (t) = R_j(t) - \int_0^t a_j(X(s))\, ds$.
Taking expectation on both sides we have 
\begin{equation}\label{eqn:LHS}
\begin{split}
&\mathbb{E}\left\{\frac{1}{t}\left(\hat{f}(X(t)) - \hat{f}(X(0)) - \int_0^t \mathcal{L}_{c^*} \hat{f}(X(s))\, ds\right)\, Z(t)\right\} \\
=& \frac{1}{t}\sum_{j = 1}^m\mathbb{E}\left\{\left[ \int_0^{\cdot} \Delta_j \hat{f}(X(s-)) \, dY_j(s), Z\right](t)\right\}\\
=& \frac{1}{tc_1} \mathbb{E}\left\{ \int_0^t \Delta_1 \hat{f}(X(s-))\, dR_1(s) \right\}\\
=& \frac{1}{tc_1} \mathbb{E}\left\{ \int_0^t \Delta_1 \hat{f}(X(s)) a_1(X(s))\, ds\right\}\PERIOD
\end{split}
\end{equation}
Note that $|a_1|_{\sqrt{V}} < \infty$ and $|\Delta_1 \hat{f}|_{V^{1/4}} < \infty$ (by Assumption~\ref{assume:intensity-bounded} and Theorem~\ref{thm:Poisson-solution}). Thus
we have $|a_1(x) \Delta_1 \hat{f}(x)| \leq K V(x)^{3/4} \leq K V(x)$ for some $K>0$. 
Hence, the last term of \eqref{eqn:LHS} converges to the ergodic limit $c_1^{-1}\pi_{c^*}(a_1\Delta_1 \hat{f})$ as $t\to\infty$.

On the other hand, again by the Poisson equation, 
\[
\left.\frac{\partial}{\partial c}\right|_{c=c^*} \pi_{c}(f)  =  \left.\frac{\partial}{\partial c}\right|_{c=c^*} \sum_{x \in E} f(x) \pi_c(x) = -\left.\frac{\partial}{\partial c}\right|_{c=c^*} \sum_{x\in E} \mathcal{L}_{c^*}\hat{f}(x) \pi_c(x)\PERIOD
\]
By definition
\begin{equation*}
 \begin{split}
     -\left.\frac{\partial}{\partial c}\right|_{c=c^*} \sum_{x\in E} \mathcal{L}_{c^*}\hat{f}(x) \pi_c(x) 
     &= -\lim_{h \to 0} \frac{1}{h}\sum_{x\in E} \mathcal{L}_{c^*}\hat{f}(x) (\pi_{c^*+h}(x) - \pi_{c^*}(x)) \\
     &= -\lim_{h \to 0}  \frac{1}{h}\sum_{x\in E} \mathcal{L}_{c^*}\hat{f}(x)\pi_{c^*+h}(x)\PERIOD
 \end{split}
\end{equation*}
We further write 
\[
-\lim_{h \to 0}  \frac{1}{h}\sum_{x\in E} \mathcal{L}_{c^*}\hat{f}(x)\pi_{c^*+h}(x) = \lim_{h \to 0}  \frac{1}{h}\sum_{x\in E} \left(-\mathcal{L}_{c^*+h}\hat{f}(x) + \mathcal{L}_{c^*+h}\hat{f}(x) - \mathcal{L}_{c^*}\hat{f}(x) \right)\pi_{c^*+h}(x)\PERIOD 
\]
Note that $\sum_{x\in E}\mathcal{L}_{c^*+h} \hat{f}(x) \pi_{c^*+h}(x) = \pi_{c^*+h}(\mathcal{L}_{c^*+h}\hat{f}) = 0$, hence
\begin{equation*}
\left.\frac{\partial}{\partial c}\right|_{c=c^*} \pi_{c}(f) = \lim_{h \to 0}  \frac{1}{h}\sum_{x\in E} \left(\mathcal{L}_{c^*+h}\hat{f}(x) - \mathcal{L}_{c^*}\hat{f}(x) \right)\pi_{c^*+h}(x)\PERIOD 
\end{equation*}
By the definition of $\mathcal{L}_{c^*}$ and $\mathcal{L}_{c^*+h}$ 
\begin{equation}\label{eqn:RHS}
\left.\frac{\partial}{\partial c}\right|_{c=c^*} \pi_c(f) = \lim_{h \to 0} \sum_{x\in E} \frac{a_1(x, c^*+h) - a_1(x, c^*)}{h} \Delta_1 \hat{f}(x) \pi_{c^*+h}(x)\PERIOD
\end{equation}

Now we compare \eqref{eqn:LHS} and \eqref{eqn:RHS},
\begin{equation}\label{eqn:Difference}
\begin{split}
&\left|  \frac{1}{c_1}\pi_{c^*}\left(a_1\Delta_1 \hat{f}\right) -  \left.\frac{\partial}{\partial c}\right|_{c=c^*} \pi_{c^*}(f)  \right|\\
=& \lim_{h \to 0}\left| \sum_{x\in E} \left(\frac{1}{c_1}a_1(x, c^*) - \frac{1}{h}(a_1(x, c^*+h) - a_1(x, c^*))\right) 
\Delta_1 \hat{f}(x) \pi_{c^*+h}(x) \right|\\
&+\lim_{h \to 0}\left|\sum_{x\in E} \frac{1}{c_1}a_1(x, c^*)\Delta_1 \hat{f}(x) (\pi_{c^*}(x)-\pi_{c^*+h}(x)) \right|\\
=& \lim_{h \to 0}\left|\sum_{x\in E} \frac{1}{c_1}a_1(x, c^*)\Delta_1 \hat{f}(x) (\pi_{c^*}(x)-\pi_{c^*+h}(x)) \right|,
\end{split}
\end{equation}
where the second equality holds since
\[
\frac{1}{h}(a_1(x, c^*+h) - a_1(x, c^*)) = \frac{1}{c_1}a_1(x, c^*)\COMMA 
\]
due to the simplifying choice of $a_j(x, c) = c_j b_j(x)$.
It remains to show that the right-hand side of \eqref{eqn:Difference} is zero.
By Assumption~\ref{assume:generator} and Theorem $3.3$ in \cite{liu2015perturbation} it follows that
\[ 
\lim_{h \to 0}\|\pi_{c^*}-\pi_{c^*+h} \|_V = 0\PERIOD 
\]
Finally, we note  that from $|a_1(x, c^*) \Delta_1 \hat{f}(x)| \leq K V(x)$ and thus the last term of \eqref{eqn:Difference} is zero, which is the desired result.
\end{proof}

\begin{remark}
{\rm 
According to the argument below \eqref{eqn:LHS} the quantity $c_1^{-1}\pi_{c^*}(a_1\Delta_1\hat{f})$ provides
a direct formula to compute the sensitivity we want. However, this formula involves the solution to
the Poisson equation $\hat{f}$ which is not analytically available. 
Nevertheless, if one has a numerical solver for the Poisson equation the sensitivity can be 
computed using the ergodic average of the observable $c_1^{-1}a_1\Delta_1\hat{f}$.
}
\end{remark}

\begin{remark}\label{remark:no-ergodic-sensitivity}
{\rm 
We comment that the convergence proved in this theorem is in expectation, 
meaning that the ensemble average has to be used to estimate the sensitivity
$\left.\frac{\partial}{\partial c}\right|_{c=c^*} \pi_c(f)$. 
However, one may be tempted to simply use the ergodic type average
\[\frac{1}{t}\int_0^t f(X(s))ds\, Z(t) \to \left.\frac{\partial}{\partial c}\right|_{c=c^*} \pi_c(f)\]
since this is a typical procedure
for estimating $\pi_{c^*}(f)$ as suggested by \eqref{eqn:ergodic-average}. 
Unfortunately, the results in the next section show that this is not possible
for the LR approach where the ergodic type limit does not hold for sensitivity estimation. 
}
\end{remark}
%%%%%%%%%%%%%%%%%%%%%%%%%%%%%%%%%%%%%%%%%%%%%%%%%%%%%%%%%%%%%%
%%%% SECTION - WEAK CONVERGENCE
%%%%%%%%%%%%%%%%%%%%%%%%%%%%%%%%%%%%%%%%%%%%%%%%%%%%%%%%%%%%%%
\section{Limiting distributions of the LR estimators}\label{sec:weak-limit}
In this section we compare the four LR estimators from the weak convergence perspective. 
The techniques of the proofs are standard and hence similar to those for the DTMC setting in \cite{glynn2017likelihood}, i.e., 
an application of the martingale functional central limit theorem \cite{whitt2007proofs, ethier2009markov}.

% We expect that estimators with well-defined weak limits are often numerically more stable than the ones 
% that do not have a limiting distribution. 

We define the matrix of the covariance function 
$$
C(t) = \left[ \begin{array}{cc}
        \sigma_{11}^2(t)& \sigma_{12}^2(t) \\
        \sigma_{21}^2(t)& \sigma_{22}^2(t)
        \end{array}
        \right]  \COMMA
$$
where  
\begin{equation}
\begin{split}
\sigma_{11}^2(t) &= t \sum_{j =1}^m \pi(a_j |\Delta_j \hat{f}|^2 ) = t \sum_{j=1}^m \sum_{x\in E} a_j(x) |\Delta_j \hat{f}|^2 \, \pi(x)\COMMA \\
\sigma_{12}^2(t) &= \sigma_{21}^2(t) = t \pi(b_1 \Delta_1 \hat{f} ) =  t \sum_{x\in E} b_1(x) \Delta_1 \hat{f}(x) \, \pi(x)\COMMA \\
\sigma_{22}^2(t) &= t \pi(b_1) =  t \sum_{x\in E} b_1(x)\,\pi(x)\PERIOD 
\end{split}
\end{equation}
First we show that $(M, Z)$ converges weakly to a Brownian motion with the mean zero and variance $C$. 
\begin{theorem}\label{thm:joint-weak-convergence}
Define $M^n(t) \triangleq M(n t)$ and $Z^n(t) \triangleq Z(nt)$, 
where $M$ and $Z$ are defined in \eqref{eqn:M} and \eqref{eqn:Z}, respectively.
Then
\begin{equation}
\left(\frac{1}{\sqrt{n}}M^n, \frac{1}{\sqrt{n}} Z^n \right) \Rightarrow W
\end{equation}
on $D_{\mathbb{R}^2}[0, 1]$ as $n \to \infty$, where
 $W = (W_1, W_2)^T$ is a two-dimensional Brownian motion with the mean zero and the covariance function $C(t)$.
\end{theorem}
\begin{proof}
We verify the three conditions of Theorem $2.1$(\romannumeral 2) in \cite{whitt2007proofs}.

{\it Maximum squared jumps of the process.}
The average jump size of the component $Z$ is negligible since  
\[
\lim_{n\to\infty}\frac{1}{n}\mathbb{E}\left\{\sup_{t \in [0, 1]}|Z^n(t) - Z^n(t-)|^2 \right\}=\lim_{n\to\infty} \frac{1}{nc_1^2} = 0\PERIOD 
\]
To verify that the average jump size of the component $M^n$ is also negligible, we write
\begin{equation}\label{eqn:max-jump-eatimate-0}
\begin{split}
&\lim_{n\to\infty}\frac{1}{n}\mathbb{E}\left\{\sup_{t \in [0, 1]}|M^n(t) - M^n(t-)|^2 \right\} \\
= &\lim_{n\to\infty}\frac{1}{n}\mathbb{E}\left\{\sup_{t \in [0, 1]}|\hat{f}(X(nt)) - \hat{f}(X(nt-))|^2 \right\}\\
\leq & \lim_{n\to\infty} \frac{1}{n}\mathbb{E}\left\{ \sup_{t \in [0, n]} |\hat{f}(X(t)) - \hat{f}(X(t-))|^2\right\}\\
\leq & \lim_{n\to\infty} \frac{1}{n}\mathbb{E}\left\{\sqrt{n} + \sup_{t\in[0,n]} |\Delta \hat{f}(X(t))|^2 \mathbbm{1}_{\{|\Delta \hat{f}(X(t))|^2 > \sqrt{n}\}}    \right\}\\
= & \lim_{n\to\infty} \frac{1}{n}\mathbb{E}\left\{\sup_{t\in[0,n]} |\Delta \hat{f}(X(t))|^2 \mathbbm{1}_{\{|\Delta \hat{f}(X(t))|^2 > \sqrt{n}\}}    \right\}\PERIOD
\end{split}
\end{equation}
Note that $\hat{f}(X(t))$ is a piecewise constant function and the jump size is 
\[
\Delta \hat{f}(X(t-)) = \Delta_j \hat{f}(X(t-))=\hat{f}(X(t-) + \nu_j) - \hat{f}(X(t-))
\]
if $R_j$ fires at $t$. 
Hence, for any $t \in [0, n]$ we have the estimate 
\begin{equation}\label{eqn:max-jump-estimate-1}
|\Delta \hat{f}(X(t))|^2 \mathbbm{1}_{\{|\Delta \hat{f}(X(t))|^2 > \sqrt{n}\}} \leq \sum_{j=1}^m|\Delta_j \hat{f}(X(t-))|^2 \mathbbm{1}_{\{|\Delta_j \hat{f}(X(t-))|^2 > \sqrt{n}\}} \PERIOD 
\end{equation}
Furthermore, we have that for each $j = 1, \ldots, m$
\begin{equation}\label{eqn:max-jump-estimate-2}
\sup_{t \in [0,n]}|\Delta_j \hat{f}(X(t-))|^2 \mathbbm{1}_{\{|\Delta_j \hat{f}(X(t-))|^2>\sqrt{n}\}} \leq \int_0^n |\Delta_j \hat{f}(X(s-))|^2 \mathbbm{1}_{\{|\Delta_j \hat{f}(X(s-))|^2 > \sqrt{n}\}} dR_j(s)\PERIOD 
\end{equation}
Combining \eqref{eqn:max-jump-estimate-1} and \eqref{eqn:max-jump-estimate-2} the right-hand side of \eqref{eqn:max-jump-eatimate-0} can be bounded as  
\begin{equation*}
\begin{split}
&\lim_{n\to\infty} \frac{1}{n}\mathbb{E}\left\{\sup_{t\in[0,n]} |\Delta \hat{f}(X(t))|^2 \mathbbm{1}_{\{|\Delta \hat{f}(X(t))|^2 > \sqrt{n}\}}    \right\}\\
\leq & 
\sum_{j=1}^m \lim_{n \to \infty} \frac{1}{n}\mathbb{E}\left\{   \int_0^n |\Delta_j \hat{f}(X(s-))|^2 \mathbbm{1}_{\{|\Delta_j \hat{f}(X(s-))|^2 > \sqrt{n}\}} dR_j(s)    \right\}\\
 = & 
\sum_{j=1}^m\lim_{n \to \infty} \frac{1}{n}\mathbb{E}\left\{   \int_0^n |\Delta_j \hat{f}(X(s))|^2 \mathbbm{1}_{\{|\Delta_j \hat{f}(X(s))|^2 > \sqrt{n}\}} a_j(X(s)) ds    \right\}\\
\leq &
\sum_{j=1}^m\lim_{n \to \infty} \frac{1}{n}\mathbb{E}\left\{   \int_0^n |\Delta_j \hat{f}(X(s))|^2 \mathbbm{1}_{\{|\Delta_j \hat{f}(X(s))|^2 > \sqrt{N}\}} a_j(X(s)) ds    \right\}\\
= & \sum_{j=1}^m \pi\left( |\Delta_j \hat{f}|^2 \mathbbm{1}_{\{|\Delta_j \hat{f}|^2 > \sqrt{N}\}} a_j   \right)
\end{split}
\end{equation*}
for any fixed integer $N$. We have used the fact that $||\Delta_j\hat{f}|^2 a_j|_{V} < \infty$ to conclude the ergodic limit in the last step. 
Taking $N \to \infty$ on the both sides of the above inequality we have, by the monotone convergence theorem,
\[
\lim_{n\to\infty} \frac{1}{n}\mathbb{E}\left\{\sup_{t\in[0,n]} |\Delta \hat{f}(X(t))|^2 \mathbbm{1}_{\{|\Delta \hat{f}(X(t))|^2 > \sqrt{n}\}}    \right\}
\leq \sum_{j=1}^m \lim_{N \to \infty}  \pi\left(  |\Delta_j \hat{f}|^2 \mathbbm{1}_{\{|\Delta_j \hat{f}|^2 > \sqrt{N}\}} a_j  \right) = 0\COMMA 
\]
%by monotone convergence theorem 
and hence 
\[
\lim_{n\to\infty}\frac{1}{n}\mathbb{E}\left\{\sup_{t \in [0, 1]}|M^n(t) - M^n(t-)|^2 \right\} = 0\PERIOD 
\]

{\it Weak convergence of the predictable quadratic covariations.}
Since the integral part $\int_0^{nt} \mathcal{L}\hat{f}(X(s))\, ds$ is adapted with continuous 
paths of finite variation, both its quadratic variation and quadratic covariation with $\hat{f}(X(nt))$
are zero (see Chapter II, Theorem $26$, in \cite{protter2005stochastic}).  
Hence, by the polarization identity of quadratic variation
\begin{equation*}
[M^n, M^n](t) = [\hat{f}(X(n\cdot)), \hat{f}(X(n\cdot))](t)\PERIOD
\end{equation*}
Applying the It\^{o} formula for jump processes, $\hat{f}(X(nt))$ can be expanded as
\[
\hat{f}(X(nt)) = \hat{f}(X(0)) + \sum_{j=1}^m \int_0^{nt} \Delta_j \hat{f}(X(s-)) \, dR_j(s)\COMMA 
\]
and hence its quadratic variation is 
\begin{equation*}
\begin{split}
[\hat{f}(X(n\cdot)), \hat{f}(X(n\cdot))](t) &= \sum_{j = 1}^m \int_0^{nt} |\Delta_j \hat{f}(X(s-))|^2 \, d[R_j, R_j](t)\\
&= \sum_{j = 1}^m \int_0^{nt} |\Delta_j \hat{f}(X(s-))|^2 \, dR_j(t)
\PERIOD 
\end{split}
\end{equation*}
Now one can verify that the predictable quadratic variation of $M^n$ is
\begin{equation*}
\langle M^n, M^n \rangle(t) = \sum_{j = 1}^m \int_0^{nt} |\Delta_j \hat{f}(X(s))|^2  a_j(X(s)) \, ds\PERIOD
\end{equation*}
Hence, 
\begin{equation*}
\frac{1}{n}\langle M^n, M^n \rangle(t) = t\sum_{j = 1}^m \frac{1}{nt} \int_0^{nt} |\Delta_j \hat{f}(X(s))|^2 a_j(X(s)) \, ds  \to \sigma_{11}^2(t)
\end{equation*}
almost surely as $n \to \infty$.
Similarly, the predictable quadratic variation of $Z^n$ is
\[
\frac{1}{n}\langle Z^n, Z^n\rangle(t) = \frac{1}{c_1n} \int_0^{nt} a_1(X(s))\, ds  \to \sigma_{22}^2(t)\COMMA 
\]
and the predictable quadratic covariation is
\begin{equation*}
\frac{1}{n}\langle M^n, Z^n \rangle(t) = t\frac{1}{c_1nt} \int_0^{nt} \Delta_1 \hat{f}(X(s)) a_1(X(s))\, ds \to \sigma_{12}^2(t) = \sigma_{21}^2(t)
\end{equation*}
almost surely as $n \to \infty$.

{\it Maximum jumps in predictable quadratic covariations.}
Since all three predictable quadratic (co)variations are continuous their jumps are always zero.
Hence,
%\[J\left(\frac{1}{n} \langle M^n(t), M^n(t) \rangle(t)\right) = J\left(\frac{1}{n} \langle Z^n, Z^n \rangle(t)\right) = J\left(\frac{1}{n} \langle M^n, Z^n \rangle(t)\right) = 0.\]
the expected value of the maximum jump in the predictable quadratic covariation is negligible.

Therefore, the conditions of Theorem $2.1$ in \cite{whitt2007proofs} are satisfied and hence the 
conclusion holds. 
\end{proof}

Now we are in position to derive the main result of this section and the limiting distributions of the four LR estimators. %The weak convergence of the four LR estimators are based on the following weak convergence theorem.  
\begin{theorem}\label{thm:sensitivity-weak-convergence}
We have
\[
\left(\frac{1}{\sqrt{n}}\int_0^{n(\cdot)} (f(X(s)) - \pi(f)) ds, \frac{1}{\sqrt{n}}Z^n\right) \Rightarrow W
\]
on $D_{\mathbb{R}^2}[0, 1]$ as $n \to \infty$, where $W$ is the same Brownian motion as in Theorem~\ref{thm:unbiased-estimator}.
\end{theorem}

\begin{proof}
First we note that 
\begin{equation*}
\begin{split}
\left(\frac{1}{\sqrt{n}}\int_0^{nt} (f(X(s)) - \pi(f)) ds,\frac{1}{\sqrt{n}}Z^n(t)\right) 
= &\left(\frac{1}{\sqrt{n}}M^n(t), \frac{1}{\sqrt{n}}Z^n(t)\right) \\
&- \left(\frac{1}{\sqrt{n}}(\hat{f}(X(nt)) - \hat{f}(X(0)), 0\right)\COMMA 
\end{split}
\end{equation*}
where the first term in the right-hand side $(n^{-1/2}M^n(t), n^{-1/2}Z^n(t)) \Rightarrow W$ by Theorem~\ref{thm:joint-weak-convergence}.
Hence, by Slutsky's lemma, it remains to show that the residual
\[
\left(\frac{1}{\sqrt{n}}(\hat{f}(X(nt)) - \hat{f}(X(0)), 0\right)
\]
converges in distribution to zero on $D_{\mathbb{R}^2}[0, 1]$.
Since the second component is a zero constant, it is sufficient to show 
the weak convergence for the first component.
We will show that the sequence (in terms of $n$) of processes ${n}^{-1/2}|\hat{f}(X(nt))|$ is tight so that convergence in finite dimensions implies convergence on $D_{\mathbb{R}}[0, 1]$.

To justify the tightness we verify the two classical conditions of Theorem~$13.2$ in \cite{billingsley2013convergence}.
For the first condition, by Markov's inequality, for any constant $a > 0$
and $t \in [0,1]$,
\[
\mathbb{P}\left(\frac{1}{\sqrt{n}}|\hat{f}(X(nt))| \geq a\right) \leq \frac{1}{na^2}\mathbb{E}\left\{  \hat{f}(X(nt))^2    \right\}\PERIOD
\]
Recalling that $\mathbb{E}\{\hat{f}(X(t))^2\} \to \pi (|\hat{f}|^2) < \infty$ it follows that
\[
\lim_{a\to\infty}\lim_{n\to \infty}\mathbb{P}\left(\frac{1}{\sqrt{n}}|\hat{f}(X(nt))| \geq a\right) = 0, 
\]
which is the first condition of Theorem~$13.2$ in \cite{billingsley2013convergence}.
For the second condition we 
note that $X(nt)$ is piecewise constant in $t$ and its jump size is independent of $n$, hence the change of $\frac{1}{\sqrt{n}}|\hat{f}(X(nt))|$ in terms of $t$ can be arbitrarily small for large $n$.
Therefore, the second condition holds trivially and it concludes the desired
weak convergence for $\frac{1}{\sqrt{n}}|\hat{f}(X(nt))|$ on $D_{\mathbb{R}}[0, 1]$.
\end{proof}

%%-----------------------------------------------------
%%WEAK LIMIT CLR
%%-----------------------------------------------------
The following result suggests that the LR estimator does not have a weak limit, which is a straightforward consequence of the above theorem.

\begin{corollary}\label{cor:LR-weak-limit}
The rescaled LR estimator $n^{-1/2}\mathcal{S}_{\LR}(n)$ converges weakly to \\
$\pi(f) W_2(1)$,
i.e., 
\[\frac{1}{\sqrt{n}}\mathcal{S}_{\LR}(n) = n^{-3/2} \int_0^n f(X(s)) ds \,Z(n) \Rightarrow \pi(f) W_2(1)\]
as $n \to \infty$, 
where $W_2$ is the second component of the Brownian motion $W$ as in Theorem~\rm{\ref{thm:joint-weak-convergence}}.
\end{corollary}

\begin{proof}
To begin, we note that by Theorem~\ref{thm:V-ergodic} $X(t)$ is positive Harris recurrent and $\pi(V) < \infty$.
Now since $\pi(|f|) < \infty$ the law of large numbers holds for $f(X(t))$, i.e., 
\[
n^{-1}\int_0^n f(X(s)) ds \to \pi(f) \;\;\mbox{a.s.}
\] 
By virtue of Theorem~\ref{thm:joint-weak-convergence} and continuous mapping we have
\[
n^{-1/2} Z(n) \Rightarrow W_2(1)\PERIOD 
\]
Now by Slutsky's lemma the product $n^{-3/2}\int_0^n f(X(s)) ds\, Z(n)$ converges weakly to $\pi(f) W_2(1)$.
\end{proof}

Since $n^{-1/2} \mathcal{S}_{\LR}$ converges weakly to a well-defined limit the LR estimator $\mathcal{S}_{\LR}$ 
does not have a weak limit.
This suggests that one cannot use the LR estimator for the ergodic average as we pointed out in Remark~\ref{remark:no-ergodic-sensitivity}. 
In contrast, the CLR estimator does converge weakly to a well-defined limit. 

\begin{corollary}\label{cor:CLR-weak-limit}
The CLR estimator converges weakly to $W_1(1)W_2(1)$, i.e., 
\[
\mathcal{S}_{\CLR}(n) = \frac{1}{n}\int_0^n (f(X(s)) - \pi(f)) ds\, Z(n)\Rightarrow W_1(1) W_2(1)
\]
as $n \to \infty$, where $W = (W_1, W_2)$ is the same Brownian motion as in Theorem~\ref{thm:joint-weak-convergence}.
\end{corollary}

\begin{proof}
By Theorem~\ref{thm:sensitivity-weak-convergence} and continuous mapping we have
\[
\frac{1}{n}\int_0^{nt} (f(X(s)) - \pi(f)) ds\, Z^n(t)\Rightarrow W_1(\cdot)W_2(\cdot)\PERIOD 
\]
Taking $t = 1$ leads to the result.
\end{proof}

We can show a similar weak convergence result for the two integral type estimators.
The proof is essentially the same as that of the discrete case in \cite{glynn2017likelihood}, which relies on Theorem $2.7$ in \cite{kurtz1991weak}.

%%-----------------------------------------------------
%%WEAK LIMIT INTEGRAL-LR
%%-----------------------------------------------------
\begin{corollary}\label{cor:integral-LR-weak-limit}
The rescaled integral type LR estimator $n^{-1/2}\tilde{\mathcal{S}}_{\LR}(n)$ converges weakly to the random variable
$\pi(f) \int_0^1 (1-s) dW_2(s)$,
i.e., 
\[
\frac{1}{\sqrt{n}}\tilde{\mathcal{S}}_{\LR}(n) = n^{-3/2} \int_0^n f(X(s)) Z(s)ds  \Rightarrow \pi(f) \int_0^1 (1-s) dW_2(s)
\]
 as $n\to \infty$.
\end{corollary}

\begin{proof}
For the ease of analysis we write the integral type LR estimator in a slightly different form.
Note that
\[
\int_0^n f(X(s)) Z(s) ds = \int_0^n \int_0^s f(X(s)) dZ(u)\, ds 
=\int_0^n \int_u^n f(X(s)) ds\, dZ(u)\COMMA 
\]
where we have changed the order of integration at the last equality.
Substituting $u = nv$ yields
\[
\int_0^n \int_u^n f(X(s)) ds\, dZ(u) = \int_0^1 \int_{nv}^n f(X(s)) ds\, dZ^n(v)\COMMA
\]
where $Z^n(t) = Z(nt)$ as we defined in Theorem~\ref{thm:joint-weak-convergence}.
Hence
\[
\frac{1}{\sqrt{n}}\tilde{\mathcal{S}}_{\LR}(n) = n^{-3/2}\int_0^1 \int_{nv}^n f(X(s)) ds\, dZ^n(v)\PERIOD
\]
For any $t \in [0, 1]$ we define $F^n(t) \triangleq \int_{nt}^n f(X(s)) ds$ for each $n \in \mathbb{Z}_+$ and $F(t) \triangleq (1-t)\pi(f)$.
We write
\[
\frac{1}{n}F^n(t) = F(t) + \frac{1}{n}\int_{nt}^n (f(X(s)) - \pi(f)) ds\PERIOD
\]
By Theorem~\ref{thm:sensitivity-weak-convergence} and continuous mapping 
the process $n^{-1/2}\int_{nt}^n (f(X(s)) - \pi(f)) ds$
converges weakly to $W_1(1) - W_1$ on $D_{\mathbb{R}}[0, 1]$ as $n \to \infty$. 
It follows that the process $n^{-1}\int_{nt}^n (f(X(s)) - \pi(f)) ds$
converges weakly to the zero function on $D_{\mathbb{R}}[0, 1]$ as $n \to \infty$.
Hence
\[
\frac{1}{n}F^n \Rightarrow F \;\;\;\mbox{on $D_{\mathbb{R}}[0, 1]$ as $n \to \infty$.}
\]
%on $D_{\mathbb{R}}[0, 1]$ as $n \to \infty$.
% Now by law of large number
% \[\frac{1}{n}F^n(t) = \frac{1}{n}\int_0^n f(X(s)) \,ds - \frac{1}{n}\int_0^{nt} f(X(s)) \, ds \to F(t)\]
% with probability one as $n\to\infty$.
Also by Theorem~\ref{thm:joint-weak-convergence} 
\[
\frac{1}{\sqrt{n}}Z^n \Rightarrow W_2 \;\;\;\mbox{on $D_{\mathbb{R}}[0, 1]$ as $n \to \infty$.}
\]
%on $D_{\mathbb{R}}[0, 1]$ as $n \to \infty$. 
Since $F$ is deterministic we have the joint convergence
\[
(n^{-1}F^n, n^{-1/2}Z^n) \Rightarrow (F, W_2)
\]
on $D_{\mathbb{R}^2}[0, 1]$ as $n\to \infty$. 
It remains to verify the conditions in Theorem $2.7$ in \cite{kurtz1991weak}.
Following the same argument as in \cite{glynn2017likelihood},
for each $\alpha > 0$ we choose the stopping time $\tau_n^{\alpha} = 2 \alpha$ so that 
$\mathbb{P}(\tau_n^{\alpha} \leq \alpha) \leq \alpha^{-1}$ is trivially satisfied. 
Moreover, 
\begin{equation}
\begin{split}
\sup_n \mathbb{E}\{[n^{-1/2}Z^n, n^{-1/2}Z^n](t \wedge \tau_n^{\alpha})\} 
&= \sup_n \frac{1}{n c_1^2}\mathbb{E}\{R_1(n(t \wedge \tau_n^{\alpha}))\}\\
&\leq \sup_n \frac{1}{n c_1^2} \int_0^{2\alpha n}\mathbb{E}\{a_1(X(s))\} ds\PERIOD
\end{split}
\end{equation}
Since by \eqref{eqn:intensity-ergodic} 
$(2\alpha n)^{-1}\int_0^{2\alpha n}\mathbb{E}\{a_1(X(s))\} ds \to \pi(a_1)$
the above supremum is finite. 
Therefore, the conditions of Theorem $2.7$ in \cite{kurtz1991weak} are justified and hence
\[
\left(n^{-1}F^n, n^{-1/2} Z^n, n^{-3/2}\int_0^1 F^n(s) dZ^n(s) \right) \Rightarrow \left(F, W_2, \int_0^1 F(s)\, dW_2(s)\right)
\]
on $D_{\mathbb{R}^3}[0, 1]$ as $n \to \infty$.
The weak convergence along the third component implies the desired result. 
\end{proof}

%%-----------------------------------------------------
%%WEAK LIMIT INTEGRAL-CLR
%%-----------------------------------------------------
Following a similar argument we can show the integral type CLR estimator converges in distribution. 
\begin{corollary}\label{cor:integral-CLR-weak-limit}
The integral type CLR estimator converges weakly to the random variable $\int_0^1 (W_1(1) - W_1(s)) dW_2(s)$, i.e., 
\[
\tilde{\mathcal{S}}_{\CLR}(n) \Rightarrow \int_0^1 (W_1(1) - W_1(s)) dW_2(s)
\]
 as $n\to \infty$.
\end{corollary}

\begin{proof}
Following the same argument as in the proof of the last theorem we write
\[
\tilde{\mathcal{S}}_{\CLR}(n) = \frac{1}{n}\int_0^1 \int_{nv}^n (f(X(s)) - \pi(f)) ds \, dZ^n(v)\PERIOD 
\]
For any $t \in [0, 1]$ we define $G^n(t) \triangleq \int_{nt}^n (f(X(s)) - \pi(f)) ds$ for each $n \in \mathbb{Z}_+$.
By Theorem~\ref{thm:sensitivity-weak-convergence} and continuous mapping
\[
(n^{-1/2}G^n, n^{-1/2}Z^n) \Rightarrow (W_1(1) - W_1, W_2)
\]
on $D_{\mathbb{R}^2}[0, 1]$ as $n\to \infty$.
Finally, by the same argument as in the proof of Corollary~\ref{cor:integral-LR-weak-limit} 
\[
\left(n^{-1/2}G^n, n^{-1/2}Z^n, n^{-1}\int_0^1 G^n dZ^n(s)\right) \Rightarrow \left(W_1(1) - W_1, W_2, \int_0^1 W_1(1) - W_1(s) dW_2(s)\right)
\]
on $D_{\mathbb{R}^3}[0, 1]$ as $n \to \infty$.
\end{proof}

%%%%%%%%%%%%%%%%%%%%%%%%%%%%%%%%%%%%%%%%%%%%%%%%%%%%%%%
%%%%%%% SECTION - Variance Analysis
%%%%%%%%%%%%%%%%%%%%%%%%%%%%%%%%%%%%%%%%%%%%%%%%%%%%%%%
\section{Variance analysis of LR estimators}\label{sec:var-analysis}

%\subsection{The growth of variance}
In this section we study
the variance of the LR estimators in terms of time.
It has been numerically observed in the literature \cite{arampatzis2016efficient} that 
the variance of the CLR estimator is independent of time. However, to the best of our knowledge, there do not exist rigorous results that justify this observation.
We provide a rigorous explanation for this behavior of the LR estimators.
%when intensities are uniformly bounded.
Throughout this section we use $C$ to denote a generic constant in order to simplify the notation. 

%%------------------------------------------------
%% ORDER OF CLR VARIANCE
%%------------------------------------------------
\begin{theorem}\label{thm:CLR-var}
Suppose all intensities $a_j$ are uniformly bounded. Then
there exists a constant $K>0$ such that
\[
\limsup_{t \to \infty}\VAR\{\mathcal{S}_{\CLR}(t)\} \leq K\PERIOD 
\]
That is, the variance of the CLR estimator is at most of order $\mathcal{O}(1)$.
\end{theorem}

\begin{remark}
{\rm 
We remark that other than the assumptions (essentially Assumption~\ref{assume:intensity-bounded}) of Section~\ref{sec:LR-sensitivity} here we additionally 
assume that all intensities $a_j$ are uniformly bounded in order to carry through the proof.
However, if we replace Assumption~\ref{assume:intensity-bounded} by a more restrictive regularity condition such as $|a_j|_{V^{1/4}} < \infty$, then we can omit the uniform boundedness assumption on $a_j$ in the above theorem and hence
our result still holds for systems with unbounded intensities. 
}
\end{remark}

\begin{proof}
Note that the expectation of $\mathcal{S}_{\CLR}(t)$ (i.e., the sensitivity) is uniformly bounded in $t$ by Theorem~\ref{thm:unbiased-estimator}.
Hence it is sufficient to show that the second moment of $\mathcal{S}_{\CLR}(t)$
is uniformly bounded in $t$.
We recall that
\[
\mathcal{S}_{\CLR}(t) = \frac{1}{t}M(t)Z(t) - \frac{1}{t}\left(\hat{f}(X(t)) - \hat{f}(X(0))\right) Z(t)\COMMA 
\]
where $M(t) = \hat{f}(X(t)) - \hat{f}(X(0)) - \int_0^{t} \mathcal{L}\hat{f}(X(s)) \,ds$ as defined in \eqref{eqn:M}.
Now we observe that the second moment of $t^{-1}(\hat{f}(X(t)) - \hat{f}(X(0))) Z(t)$ vanishes as $t \to \infty$, which can be shown easily
using the same argument
as in the proof of Theorem~\ref{thm:unbiased-estimator}. 
Therefore, we only show how to bound the second moment of $t^{-1}M(t)Z(t)$.
First, by the Cauchy-Schwartz inequality we have
\[
t^{-2}\mathbb{E}\{M^2(t)Z^2(t)\} \leq t^{-2}\sqrt{\mathbb{E}M^4(t) \mathbb{E}Z^4(t)}\PERIOD 
\]
Since $M(t)$ is a martingale it follows by the BDG inequality
\[
\mathbb{E}M^4(t) \leq C \mathbb{E}\left\{\left(\sum_{j = 1}^m \int_0^{t} |\Delta_j \hat{f}(X(s-))|^2 \, dR_j(t)\right)^2 \right\}\PERIOD 
\]
Next we recall $R_j(t) = Y_j(t) + \int_0^t a_j(X(s)) ds$ and thus
\begin{equation*}
\begin{split}
\mathbb{E}M^4(t) &\leq m \sum_{j=1}^m 
\mathbb{E}\left\{\left(\int_0^t |\Delta_j \hat{f}(X(s-))|^2 \, dR_j(s)\right)^2\right\}\\
&\leq 2m \sum_{j=1}^m \mathbb{E} \left\{\left(\int_0^t |\Delta_j \hat{f}(X(s-))|^2 \, dY_j(s)\right)^2
+ 
\left(\int_0^t |\Delta_j \hat{f}(X(s))|^2 \, a_j(X(s))ds\right)^2\right\}\PERIOD 
\end{split}
\end{equation*}
We estimate the two expectations at the right hand side separately. 
For the first expectation we apply the BDG inequality again to obtain
\begin{equation}\label{eqn:var-CLR-estimate1}
\begin{split}
\mathbb{E} \left\{\left(\int_0^t |\Delta_j \hat{f}(X(s-))|^2 \, dY_j(s)\right)^2\right\} &\leq 
C\mathbb{E} \left\{\int_0^t |\Delta_j \hat{f}(X(s-))|^4 dR_j(s)\right\}\\
& \leq C\mathbb{E}\left\{\int_0^t |\Delta_j \hat{f}(X(s))|^4 a_j(X(s)) ds\right\}\\
&\leq C \mathbb{E}\left\{\int_0^t |\Delta_j \hat{f}(X(s))|^4  ds\right\}\COMMA
\end{split}
\end{equation}
where the last inequality holds since intensities are uniformly bounded.
%Now recall that $|\Delta_j \hat{f}(x)| \leq C V(x)$ by Theorem \ref{thm:V-ergodic}, it follows that
%$\lim_t t^{-1}\mathbb{E}\{\int_0^t |\Delta_j \hat{f}(X(s))|^4  ds\}$
%and hence $\mathbb{E}\{\int_0^t |\Delta_j \hat{f}(X(s))|^4  ds\} = \mathcal{O}(t)$. 
As for the second expectation, by Jensen's inequality we obtain 
\[
\left(\int_0^t |\Delta_j \hat{f}(X(s-))|^2 \, a_j(X(s))ds\right)^2
\leq t \int_0^t |\Delta_j \hat{f}(X(s-))|^4 a_j^2(X(s)) ds\COMMA
\]
and hence
\begin{equation}\label{eqn:var-CLR-estimate2}
\mathbb{E} \left\{
\left(\int_0^t |\Delta_j \hat{f}(X(s-))|^2 \, a_j(X(s))ds\right)^2\right\}
\leq C t \mathbb{E}\left\{\int_0^t |\Delta_j \hat{f}(X(s-))|^4 ds\right\}\PERIOD
\end{equation}
Combining \eqref{eqn:var-CLR-estimate1} and \eqref{eqn:var-CLR-estimate2} 
and the fact that $|\Delta_j \hat{f}(x)| \leq C V(x)$ we have
\begin{equation}\label{eqn:var-CLR-EM4}
    \mathbb{E}M^4(t) \leq 2m C (t + 1) \sum_{j=1}^m \mathbb{E}\left\{\int_0^t |\Delta_j \hat{f}(X(s))|^4 ds\right\}
    \leq 2m^2 C (t + 1) \mathbb{E}\left\{\int_0^t V(X(s)) ds\right\}\PERIOD
\end{equation}
Using exactly the same argument we can bound $\mathbb{E}Z^4(t)$ by
\begin{equation*}\label{eqn:var-CLR-EZ4}
    \mathbb{E}Z^4(t) \leq C(t^2 + t)\PERIOD
\end{equation*}
Therefore, we have
\begin{equation*}
\begin{split}
    t^{-2}\sqrt{\mathbb{E}M^4(t) \mathbb{E}Z^4(t)} \leq C\sqrt{1+t^{-1}} \sqrt{(t^{-1} + t^{-2})\mathbb{E}\left\{\int_0^t V(X(s)) ds\right\} }\PERIOD
\end{split}
\end{equation*}
Taking $t \to \infty$ on both sides and using the fact that 
$\lim_{t\to \infty} t^{-1} \mathbb{E}\{\int_0^t V(X(s)) ds\} = \pi_{c^*}(V)$ (by Theorem~\ref{thm:V-ergodic}),
\[
\limsup_{t \to \infty}t^{-2}\sqrt{\mathbb{E}M^4(t) \mathbb{E}Z^4(t)} \leq C \sqrt{\pi_{c^*}(V)} < \infty\COMMA 
\]
which yields the desired result.
\end{proof}

%%---------------------------------------------------------------------
%% ORDER OF LR VARIANCE
%%---------------------------------------------------------------------

The order of the LR estimator variance can be obtained easily based on the above result.   
\begin{corollary}\label{cor:LR-var}
Suppose all intensities are uniformly bounded. Then there exists a constant $K>0$ such that 
\[
\limsup_{t \to \infty} \frac{1}{t}\VAR\{\mathcal{S}_{\LR}(t)\} \leq K\PERIOD 
\]
That is, the variance of the LR estimator is at most of order $\mathcal{O}(t)$.
\end{corollary}

\begin{proof}
Note that 
\[
\mathcal{S}_{\LR}(t) = \mathcal{S}_{\CLR}(t) + \pi(f) Z(t)\COMMA 
\]
and hence 
\[
t^{-1}\mathbb{E}\{\mathcal{S}_{\LR}(t)^2\} \leq 
2t^{-1}\mathbb{E}\{\mathcal{S}_{\CLR}(t)^2\} 
+ 2t^{-1}\pi(f)^2\mathbb{E}\{Z^2(t)\}\PERIOD 
\]
We recall that $\lim_{t \to \infty}t^{-1}\mathbb{E}\{Z^2(t)\}$ exists (see, proof of Theorem~\ref{thm:unbiased-estimator}) and $\mathbb{E}\{\mathcal{S}_{\CLR}(t)^2\} = \mathcal{O}(1)$ by Theorem~\ref{thm:CLR-var}. Thus
the result follows by taking $t \to \infty$ on both sides of the above inequality. 
\end{proof}

The next theorem says that the variance of the integral type CLR estimator does not increase in time. 
\begin{theorem}\label{thm:integral-CLR-var}
Suppose all intensities are uniformly bounded. Then there exists a constant $K >0$ such that 
\[
\limsup_{t \to \infty} \VAR\{\tilde{\mathcal{S}}_{\CLR}(t)\} \leq K\PERIOD 
\]
That is, the variance of the integral type CLR estimator is at most of order $\mathcal{O}(1)$.
\end{theorem}

\begin{proof}
We recall that the alternative representation of the integral type CLR estimator and the CLR estimator are
\[
\tilde{\mathcal{S}}_{\CLR}(t) =  \frac{1}{t}\int_0^t \int_u^t (f(X(s)) - \pi(f)) \,ds\, dZ(u)
\]
and 
\[
\mathcal{S}_{\CLR}(t) =  \frac{1}{t}\int_0^t \int_0^t (f(X(s)) - \pi(f)) \,ds\, dZ(u)\COMMA 
\]
respectively.
It follows that 
\[
\tilde{\mathcal{S}}_{\CLR}(t) =\mathcal{S}_{\CLR}(t) + \frac{1}{t} \int_0^t \int_0^u (f(X(s)) - \pi(f)) \,ds\, dZ(u)\PERIOD 
\]
Note that the second moment of $\mathcal{S}_{\CLR}(t)$ is $\mathcal{O}(1)$, hence 
we shall only bound the second moment of $\int_0^t \int_0^u (f(X(s)) - \pi(f)) \,ds\, dZ(u)$. 
Using the Poisson equation and rearranging the terms  we have
\[
\frac{1}{t} \int_0^t \int_0^u (f(X(s)) - \pi(f)) \,ds\, dZ(u) = 
\frac{1}{t}\int_0^t \hat{f}(X(u)) - \hat{f}(X(0)) dZ(u)
-\frac{1}{t}\int_0^t M(u) dZ(u)\COMMA
\]
where $M(t)$ is defined in \eqref{eqn:M}.
It suffices to estimate the second term on the right-hand side of the above equality. 
Since $\int_0^t M(u) dZ(u)$ is an $\mathcal{F}_t$-local martingale by the BDG inequality and 
$a_j$ is uniformly bounded we have
\begin{equation*}
\begin{split}
    \frac{1}{t^2}\mathbb{E}\left\{\left(\int_0^t M(u) dZ(u)\right)^2\right\}
    \leq \frac{1}{t^2}\mathbb{E}\left\{ \int_0^t M^2(u) dR_1(u)\right\}
    &\leq \frac{C}{t^2} \int_0^t \mathbb{E}\{M^2(u)\} du\\
    &\leq \frac{C}{t}\mathbb{E}\left\{\left(\sup_{u \leq t}|M(u)|\right)^2\right\}\PERIOD
\end{split}    
\end{equation*}
Applying the BDG inequality again
\[\frac{C}{t}\mathbb{E}\left\{\left(\sup_{u \leq t}|M(u)|\right)^2\right\} \leq \frac{C}{t}\sum_{j=1}^m\mathbb{E}\left\{\int_0^t |\Delta_j \hat{f}(X(s))|^2 ds\right\} 
\leq \frac{mC}{t} \mathbb{E}\left\{\int_0^t V(X(s))^{1/2}ds\right\}\COMMA
\]
where the last term converges to $m C \pi_{c^*}(V^{1/2})$ as $t \to \infty$.
Therefore, the variance of $\tilde{\mathcal{S}}_{\CLR}(t)$ is at most $\mathcal{O}(1)$.
\end{proof}

Based on the above result we can easily see that $\VAR\{\tilde{\mathcal{S}}_{\LR}(t)\}$ is at most $\mathcal{O}(t)$.
We omit the proof since the argument follows the same line as the proof
of Corollary~\ref{cor:LR-var}.
\begin{corollary}\label{cor:integral-LR-var}
Suppose all intensities are uniformly bounded. Then there exists some constant $K >0$ such that 
\[
\limsup_{t \to \infty} \frac{1}{t}\VAR\{\tilde{\mathcal{S}}_{\LR}(t)\} \leq K\PERIOD 
\]
That is, the variance of the integral type LR estimator is at most of order $\mathcal{O}(t)$.
\end{corollary}

%\subsection{Asymptotic variance}
Theorems~\ref{thm:CLR-var} and \ref{thm:integral-CLR-var} show that the variance of the CLR estimator $\mathcal{S}_{\CLR}(t)$ 
and the int-CLR estimator $\tilde{\mathcal{S}}_{\CLR}(t)$ do not grow in $t$
and hence they outperform the LR and int-LR estimators.  
One may wonder, by comparing the CLR and int-CLR estimators, which of them works better.
In order to answer this question we compare the variances of their limiting distributions (see Section~\ref{sec:weak-limit}), i.e., their asymptotic variances. 
Due to the fact that the limiting distributions we derived are exactly the same as those in the discrete time setting
we can cite the result of Glynn and Olvera-Cravioto \cite{glynn2017likelihood}. 

\begin{theorem}[Glynn and Olvera-Cravioto \cite{glynn2017likelihood}]\label{thm:asymptotic-var}
Let $(W_1, W_2)^T$ be the two-dimensional Brownian motion defined in Theorem~\ref{thm:joint-weak-convergence}.
Denote 
\[
\mathcal{S}_{\CLR}^{\infty} \triangleq W_1(1) W_2(1)
\]
and 
\[
\tilde{\mathcal{S}}_{\CLR}^{\infty} \triangleq \int_0^1 (W_1(1) - W_1(s)) dW_2(s)
\]
the limiting distributions of the CLR estimator $\mathcal{S}_{\CLR}(t)$ and integral type CLR estimator $\tilde{\mathcal{S}}_{\CLR}(t)$, respectively. 
Then we have
\[
\VAR\{\tilde{\mathcal{S}}_{\CLR}^{\infty}\} \leq \frac{1}{2}\VAR\{\mathcal{S}_{\CLR}^{\infty}\}\PERIOD 
\]
\end{theorem}
This result suggests that the asymptotic variance of the int-CLR estimator is at least twice lower than that of the CLR estimator. 
We will demonstrate this reduction of variance by a numerical benchmark of the next section.

%%%%%%%%%%%%%%%%%%%%%%%%%%%%%%%%%%%%%%

\section{Numerical benchmarks}\label{sec:numerics}
We illustrate the convergence and the dependence of the variance on time for the four analyzed LR estimators through two numerical examples. 
The scripts and Matlab functions that were used to create the numerical experiments described in this paper can be found at 
\url{https://github.com/ting-w/LR-sensitivity}.

\subsection{A linear system}
We consider a simple linear reaction network that consists of four reactions and three species.
The reaction channels, intensities, parameters and stoichiometric vectors are listed in Tab~\ref{tab:network}.
\begin{table}[ht]
\centering
\caption{A linear reaction network.}
\begin{tabular}{cccc}
\hline
\hline
Reaction & Intensity   & Parameter & Stoichiometric vector\\
\hline
$S_1 \xrightarrow{c_1} S_2$  & $a_1(x, c) = c_1 x_1 $  & $c_1 = 10.0$ & $\nu_1=(-1, 1, 0)^T$\\
\hline
$S_2 \xrightarrow{c_2} S_1$  & $a_2(x, c) = c_2 x_2$  & $c_2 = 20.0$ & $\nu_2 = (1, -1, 0)^T$\\
\hline
$S_2 \xrightarrow{c_3} S_3$  & $a_3(x, c) = c_3 x_2$  & $c_3= 0.03$ & $\nu_3 = (0, -1, 1)^T$\\
\hline
$S_3 \xrightarrow{c_4} S_2$  & $a_4(x, c) = c_4 x_3$  & $c_4= 0.02$ & $\nu_4 = (0, 1, -1)^T$\\
\hline
\hline
\end{tabular}
\label{tab:network}
\end{table}
We apply the LR, int-LR, CLR, and int-CLR estimators to estimate the sensitivity of $\pi(x_1)$ (i.e., the steady-state mean population of $S_1$) with respect to the parameter $c_3$.  
We set the initial population as $(5, 5, 0)^T$.
Since the intensities are all linear the sensitivity of $\pi(x_1)$ can be calculated exactly.
The estimated sensitivities from $\mathcal{S}_{\LR}(t)$, $\tilde{\mathcal{S}}_{\LR}(t)$, $\mathcal{S}_{\CLR}(t)$,
and $\tilde{\mathcal{S}}_{\CLR}(t)$
are compared with the exact sensitivity.
Here the terminal time is varied from $t = 100$ to $t=1000$ to test the convergence of the four estimators. 
For each fixed terminal time $t$ we sample $N = 10^4$ realizations for each estimator and compute the sample average.

\begin{figure}[ht]
\centering
  \includegraphics[height=5cm]{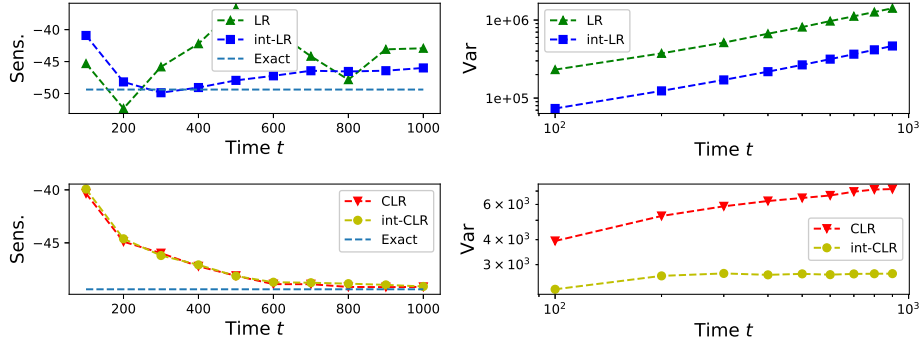}
\caption{Sensitivity estimation for ${\partial \pi(x_1)}/{  \partial c_3}$. Left: plot of the estimated sensitivity. 
Right: log-log plot of the estimated variances. }
\label{fig:sensitivity}
\end{figure}

The simulation results are shown in Fig~\ref{fig:sensitivity}. 
From the left plot of Figure~\ref{fig:sensitivity} we can see that the estimated sensitivities from the LR and int-LR estimators fluctuate even for very large terminal time $t$.
However, we see a fast convergence of the sensitivities estimated from the CLR and int-CLR estimators. 
This observation suggests that the centered estimators tend to have much smaller variance. This is also demonstrated in the right plot of Fig~\ref{fig:sensitivity}, where the growth of the estimator variance against time is shown on the log-log scale. It can be seen that the variance growth is roughly linear (with slope $1$) in $t$ for both the LR and integral type LR estimators, which is consistent with the result proved in Corollaries~\ref{cor:LR-var} and \ref{cor:integral-LR-var}. Similarly, as predicted by Theorems~\ref{thm:CLR-var} and \ref{thm:integral-CLR-var} the variance of the CLR and integral type CLR estimators are constant for large $t$. 
Finally, we observe that in the large $t$ regime the variance of the CLR estimator is more than twice larger than that of the int-CLR estimator, which 
numerically confirms Theorem~\ref{thm:asymptotic-var}.

%%%%%%%%%%%%%%%%%%%%%%%%%%%%%%%%%%%%%%%%%%%%%%%%%%%%%%%%%%%%%
\begin{table}[ht]
\centering
\caption{Reactions and intensities of the two gene complex system.}
\begin{tabular}{p{5cm}ll}
\hline
\hline
Reaction & \multicolumn{2}{c}{Intensity}    \\
\hline
$\emptyset \xrightleftharpoons[a_2]{a_1} m_A$  & $a_1 = k_r \frac{\phi^4}{\phi^4 + p_{A_2}^4},$ & $a_2 = k_{dr}m_A $  \\
$\emptyset \xrightleftharpoons[a_4]{a_3} p_A$  & $a_3 = k_p m_A,$ & $a_4 = k_{dp} p_A$  \\
$p_A + p_A \xrightleftharpoons[a_6]{a_5} p_{A_2}$  & $a_5 = k_1 p_A (p_A - 1),$ &  $a_6 = k_2 p_{A_2}$  \\
$\emptyset \xrightleftharpoons[a_8]{a_7} m_B$  & $a_7 = k_r \frac{\phi^2}{\phi^2 + p_{AB}^2},$ &  
$a_8 = k_{dr} m_B$  \\
$\emptyset \xrightleftharpoons[a_{10}]{a_9} p_B$  & $a_9 = k_p m_B ,$ & 
$a_{10} = k_{dp} p_B$  \\
$p_A + p_B \xrightleftharpoons[a_{12}]{a_{11}} p_{AB}$  & $a_{11} = k_3 p_A p_B, $ & 
$a_{12} = k_4 p_{AB}$  \\
\hline
\hline
\end{tabular}
\label{tab:six-dim-network}
\end{table}

\subsection{A two gene complex system}
To further demonstrate the performance of the sensitivity estimators we consider a complex biochemical reaction system which models the interaction between two genes $A$ and $B$ \cite{milias2014fast}.
The system contains $6$ species that are evolving according to $12$ reactions.
The reactions and their corresponding intensity functions are listed in Tab~\ref{tab:six-dim-network}.
There are $9$ parameters whose values are set as follows:
\[[k_r, \phi, k_{dr}, k_p, k_{dp}, k_1, k_2, k_3, k_4] = [1, 60, 0.1, 1, 0.5, 0.02, 0.08, 0.02, 0.1].\]
We aim to estimate the sensitivity of $\pi(\#p_{AB})$ with respect to each of the parameters and test the dependence
of the estimator variances on time. 
The initial state of the six-dimensional vector is set to 
\[[m_A(0), p_A(0), p_{A_2}(0), m_B(0), p_B(0), p_{AB}(0)] = [0, 0, 0, 0, 0, 0].\]
We test the performance of the four estimators with terminal times $t = 2.5\times10^4$ in order to approximate 
the steady-state sensitivity. 
For each fixed terminal time, we repeat $N = 10^5$ times to obtain the ensemble average for each estimator.
Note that we are able to estimate the sensitivity with respect to all $9$ parameters simultaneously, 
which is also one of the advantages of the LR approach over other sensitivity estimation methods in problems where
the parameter space is high dimensional. 

The estimated sensitivities along with their associated $95\%$ confidence intervals with terminal time $t = 2.5\times10^4$ are summarized in Tab~\ref{tab:six-dim-network-sensitivity}. 
We observe that even with the sample number as large as $N = 10^5$ the sensitivities estimated by the LR and int-LR estimators 
are still completely off with overwhelmingly large confidence intervals. However, the sensitivities estimated by the 
CLR and int-CLR estimators have very tight confidence intervals suggesting that the simulations results 
are statistically correct. 
This observation can be explained by the large variances associated with the LR and int-LR estimators which increase
linearly in time as suggested by Corollaries~\ref{cor:LR-var} and \ref{cor:integral-LR-var}. 
On the other hand, the variances of the CLR and int-CLR estimators are constant in time (see Theorem~\ref{thm:CLR-var} and Theorem~\ref{thm:integral-CLR-var}).

\begin{table}[!ht]
\centering
\setlength\tabcolsep{-100pt}
\setlength\extrarowheight{2pt}
\caption{Sensitivities and the associated confidence intervals for ${\partial \pi(\# p_{AB})}/{  \partial k_r}$ with $t = 2.5\times10^4$.}
\label{tab:six-dim-network-sensitivity}
\begin{tabular*}{\textwidth}
{
  @{\extracolsep{\fill}}
  l
  S[table-format=-2.2(5)]
  S[table-format=-2.2(5)]
  S[table-format=-2.2(5)]
  S[table-format=-2.2(5)]
  @{}
}
\hline
\hline
  & {LR} & {int-LR} & {CLR}  & {int-CLR}   \\
\hline
$k_r$       & 42.64\pm44.84      & 34.17\pm25.96      & 32.97\pm0.42    & 33.03\pm0.26\\
$\phi$ & 0.65\pm0.72        & 0.49\pm0.42        & 0.37\pm0.01     & 0.38\pm0.00 \\
$k_{dr}$    & -767.60\pm447.31   & -630.23\pm258.33   & -326.35\pm4.16  & -327.77\pm2.59\\
$k_p$       & 79.72\pm142.00     & 7.84\pm82.00       & 32.02\pm1.17    & 32.36\pm0.82\\
$k_{dp}$    & -18.92\pm283.28    & -99.84\pm163.91    & -64.65\pm2.34   & -64.48\pm1.64\\
$k_1$       & -1233.63\pm3733.53 & -760.17\pm2154.54  & -411.09\pm30.58 & -415.91\pm21.62\\
$k_2$       & 11.10\pm930.79     & 100.09\pm537.39    & 103.98\pm7.63   & 103.80\pm5.39\\
$k_3$       & 260.09\pm3726.63   & 336.88\pm2150.83   & 1212.21\pm31.14 & 1227.17\pm21.43\\
$k_4$       & 9.31\pm742.05      & -182.66\pm428.07   & -244.20\pm6.27  & -246.08\pm4.31\\
\hline
\hline
\end{tabular*}
\end{table}

To further confirm the above observation, we plot the variances of the four estimators with varying terminal times
$t = 5.0\times10^3, 1.0\times10^4, 1.5\times10^4, 2.0\times10^4$ and $2.5\times10^4$ in Fig~\ref{fig:sixD_sensitivity}. 
Here we only demonstrate the result with respect to the parameter $k_r$.
From the left plot we can see that the two centered LR estimators give consistent results even when $t = 5000$.
However, the sensitivities estimated by the LR and int-LR estimators change with time, indicating that the associated variances are large. 
This is confirmed by the variance plot (in log-log) at the right hand side of Figure~\ref{fig:sixD_sensitivity}.
The variances of LR and int-LR estimators grow linearly in time while those of CLR and int-CLR estimators remain roughly constant in time. At the time $t = 2.5\times10^4$, the centered estimators achieve a variance reduction that is up to the order of $10^4$ over the noncentered ones. 
Furthermore, the variance of int-CLR is less than half of the variance of CLR as predicted by Theorem~\ref{thm:asymptotic-var}.

\begin{figure}[ht]
\centering
  \includegraphics[height=5cm]{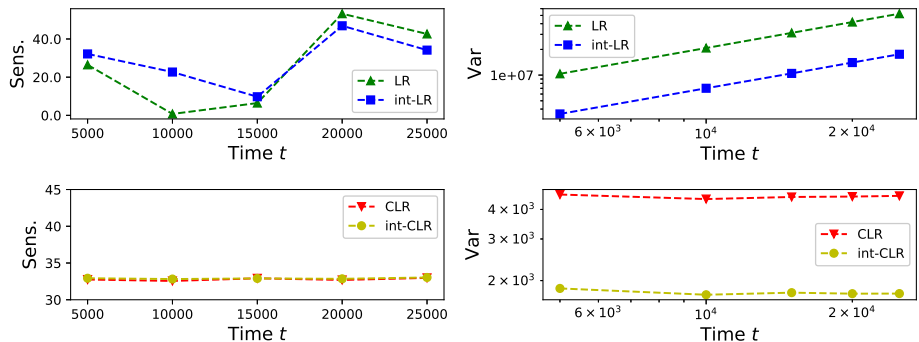}
\caption{Sensitivity estimation for ${\partial \pi(\# p_{AB})}/{  \partial k_r}$. Left: plot of the estimated sensitivity. Right: log-log plot of the estimated variances.}
\label{fig:sixD_sensitivity}
\end{figure}

%%%%%%%%%%%%%%%%%%%%%%%%%%%%%%%%%%%%%%%%%%%%%%%%%%%%%%%%%%%%%
%%%%%%%%%%%% APPENDIX
%%%%%%%%%%%%%%%%%%%%%%%%%%%%%%%%%%%%%%%%%%%%%%%%%%%%%%%%%%%%%

\appendix

\section{Burkholder--Davis--Gundy inequality}\label{app_BDG}
We include the BDG inequality for the convenience of reference.  
The case $p>1$ was established by Burkholder in \cite{burkholder1966martingale} and the case
$p=1$ was obtained by Davis in \cite{davis1970intergrability}.

{\bf The BDG inequality.} {\rm Let $Y$ be a local martingale and let $\tau$ be a stopping time. 
Define $Y^*(\tau) = \sup_{t \leq \tau}|Y(t)|$.
Then for any $p \geq 1$ there exist constants $k_p$ and $K_p$ such that
\[
k_p\mathbb{E}\{([Y, Y](\tau))^{p/2}\} \leq \mathbb{E}\{(Y^*(\tau))^p\} \leq K_p\mathbb{E}\{([Y, Y](\tau))^{p/2}\}\PERIOD 
\]
}

\section{The existence of Lyapunov function}
It can be seen from our analysis that the Foster-Lyapunov drift condition is crucial for the justification of using the 
LR method for the steady-state sensitivity analysis. 
However, the natural question is whether there exists such $V$.
Here we provide an easy-to-verify condition which allows us to construct $V$ explicitly.
We point out that this choice of $V$ may not be unique.
%We can impose the following verifiable assumption to replace Assumption~\ref{assume:drift-condition}. 

{\bf The existence of $V$.} {\rm For a positive vector $v \in \mathbb{R}^n$, 
there exists $\alpha_1 > 0$, $0<\alpha_2 < \infty$
such that 
\begin{equation}\label{eqn:verifiable-drift-condition}
    \sum_{j=1}^m a_j(x) \langle v, \nu_j \rangle \leq - \alpha_1 \langle v, x \rangle-\alpha_1+ \alpha_2 \mathbbm{1}_{\{0\}}(x)\COMMA 
\end{equation}
where $\langle \cdot, \cdot \rangle$ is the inner product on $\mathbb{R}^n$.}

It is easy to see that when we define $V(x) = 1 + \langle v, x \rangle $,
\eqref{eqn:verifiable-drift-condition} is equivalent to \eqref{eqn:V-drift} in Assumption~\ref{assume:drift-condition}.
A similar condition with \eqref{eqn:verifiable-drift-condition} is used in \cite{gupta2014scalable} to study the longtime behavior of stochastic reaction kinetics. See \cite{gupta2014scalable} for examples about how to choose the vector $v$ so that 
\eqref{eqn:verifiable-drift-condition} holds.

\section*{Acknowledgment}
%The work of T.W. has been supported by the U.S. Department of Energy, Office of Science, Office of Advanced Scientific Computing Research, Applied Mathematics program under the contract number DE-SC0010549. The work of P.P.  was partially supported by the DARPA project W911NF-15-2-0122.

The work has been partially supported by the U.S. Department of Energy, Office of Science, Office of Advanced Scientific Computing Research, Applied Mathematics program under the contract number DE-SC0010549 and by the DARPA project W911NF-15-2-0122.

%%%%%%%%%%%%%%%%%%%%%%%%%%%%%%%%%%%%%
%
% BIBLIOGRAPHY
%
%%%%%%%%%%%%%%%%%%%%%%%%%%%%%%%%%%%%%

\bibliographystyle{abbrv}
\bibliography{LR}

\end{document}